%% file: ex_article.tex
\documentclass[hidelinks,onefignum,onetabnum]{siamart251216}

\usepackage{subcaption}
\usepackage{booktabs}
\usepackage{bbm}

\input{ex_shared}

\ifpdf
\hypersetup{
  pdftitle={A weighted quantile filter based framework for interface optimal design problems},
  pdfauthor={Sihao Cheng, Ziming Shao, Dong Wang}
}
\fi

\begin{document}

\maketitle

\begin{abstract}
We present a robust and efficient numerical framework based on a median filter scheme for solving a broad class of interface optimal design problems, from image segmentation to topology optimization. A key innovation of our work is the extension of the binary scheme into a level-set scheme via a weighted quantile interpretation. Unlike traditional binary iterative convolution-thresholding method (ICTM), this continuous weighted quantile filter scheme effectively overcomes the pinning effect caused by spatial discretization, achieving interface evolution even with small time steps. We also provide a rigorous theoretical analysis, proving the unconditional energy stability of the iterative scheme. Furthermore, we prove that for a wide class of data fidelity terms, the convex relaxation inherently enforces a binary solution, justifying the effectiveness of the method without explicit penalization. Numerical experiments on the Chan--Vese model, the local intensity fitting (LIF) model, and topology optimization in Stokes flow demonstrate that the proposed efficient continuous framework effectively eliminates the pinning effect, guarantees unconditional energy stability, and accurately converges to binary solutions.
\end{abstract}



\section{Introduction}
Interface optimal design problems constitute a fundamental class of challenges in computational mathematics, permeating diverse fields ranging from computer vision to fluid mechanics. Prominent examples include the segmentation of complex structures in image processing \cite{Chan_2001, Li_2007, Ma_2021, Mumford_1989}, the tracking of immiscible multiphase fluid interfaces in hydrodynamics \cite{Hirt_1981, Jacqmin_1999, Sussman_1994}, and the optimal distribution of materials in topology optimization \cite{Allaire_2004, Bendsoe_1989, Takezawa_2010}. In these variational problems, the primary objective is to determine an optimal partition of a bounded domain $\Omega$ into distinct sub-regions, such as separating foreground from background or distinguishing solids from fluids, in order to minimize a specific, often non-convex energy functional. However, the numerical treatment of these problems causes significant challenges. The challenges stem not only from the geometric complexity of the interfaces, potentially exhibiting singularities, but also from the topological changes that occur naturally during the optimization process, such as merging and splitting \cite{Deckelnick_2005, saye2020review}. Consequently, developing numerical algorithms that are both computationally efficient and mathematically stable, while robustly handling these geometric evolutions, remains an ongoing area of research.

Traditionally, interface evolution modeling uses the level set method \cite{Osher_1988}, the phase field method \cite{Du_2020} and so on. Although these methods provide a natural framework for handling complex topological changes, they also present computational challenges. For example, the level set method often requires frequent re-initialization to maintain the properties of the signed distance function, which leads to high computational costs \cite{Hartmann_2010}. Similarly, the phase field method is governed by nonlinear high-order partial differential equations that depend on a small interface width parameter, and resolving this thin interface requires high grid resolution. Therefore, explicit schemes suffer from severe restrictions of the CFL stability condition, while implicit schemes, though more stable, require solving large nonlinear algebraic systems at each step, incurring high computational costs \cite{Li_2026, Zhang_2026}.

As an alternative, Merriman, Bence, and Osher introduced the threshold dynamics method to simulate the mean curvature flow, also known as the MBO method \cite{merriman1992diffusion, Merriman_1994}. The mathematical foundation of this approach was significantly advanced by Esedo\=glu and Otto \cite{Esedo_Lu_2014}. By introducing a novel movement limiter mechanism to decouple non-local interactions, they rigorously established the variational formulation of threshold dynamics. This formulation not only interpreted the threshold dynamics scheme as a sequence of coordinate descent steps but also provided a theory for its unconditional energy stability. The idea of diffusion generated approaches has been successfully extended to a wide range of general interface related problems, including wetting dynamics \cite{Wang_2019b, Xu_2017}, image segmentation \cite{Liu_2022, Liu_2023, Wang_2017, Wang_2022b}, topology optimization \cite{Chen_2024, Chen_2026, Chen_2022}, target-valued harmonic maps \cite{Osting_2019, Wang_2019a}, Dirichlet partition \cite{Wang_2022a}, and data clustering \cite{van_Gennip_2014}. However, MBO schemes typically suffer from the well-known pinning effect, where the interface fails to move if the time step is chosen too small relative to the mesh size, leading to superficial convergence or loss of geometric details \cite{Esedog_lu_2010, Merriman_1994}. This limitation often requires a delicate balance between spatial and temporal discretization, restricting its applicability in high-precision simulations.

In this paper, we propose a weighted quantile filter based framework to solve interface optimal design problems, inspired by the recently developed median-filter method for mean curvature flow \cite{Esedo_lu_2024}. While drawing on the efficiency of threshold dynamics, our method differs fundamentally from the traditional binary scheme. Using a weighted quantile interpretation of the heat kernel convolution, we extend the discrete thresholding step into a continuous level-set update scheme. This innovation effectively eliminates the pinning effect associated with spatial discretization while retaining the unconditional energy stability and computational efficiency of the original MBO method. In addition, we theoretically establish that, for a wide class of data fidelity terms, the proposed convex relaxation naturally enforces a binary solution, thereby recovering a clear interface without the need for explicit penalty terms. Furthermore, we demonstrate that this framework is not limited to mean curvature flow, but can be rigorously applied to minimize general energy functionals that contain complex data fidelity terms and physical constraints, such as those governed by the Stokes equations. 

The remainder of this paper is organized as follows. In Section \ref{sec:model}, we introduce the mathematical formulation of interface optimal design problems and define the associated energy functional in the space of functions of bounded variation. In Section \ref{sec:algorithm}, we derive the proposed continuous weighted quantile filter scheme. Using a weighted quantile interpretation of the heat kernel convolution, this method effectively overcomes the pinning effect inherent in traditional threshold dynamics. Section \ref{sec:theory} is dedicated to the rigorous theoretical analysis of the framework. We first prove that for a wide class of data configurations, the convex relaxation automatically enforces a binary solution, ensuring the recovery of sharp interfaces, and then establish the unconditional energy stability of the iterative scheme. Section \ref{sec:numerical_experiments} validates the versatility and robustness of the algorithm through extensive numerical experiments, ranging from image segmentation models, such as the Chan-Vese model and the local intensity fitting (LIF) model, to topology optimization in Stokes flow. Finally, concluding remarks are given in Section \ref{sec:conclusion}.

\section{Mathematical Model}
\label{sec:model}

In this section, we introduce a mathematical model for a class of interface related optimization problems. The objective is to find an optimal partition of a bounded domain $\Omega \subset \mathbbm{R}^{2}$ and the corresponding optimal parameters that characterize the sub-regions. To ensure the well-posedness of the problem, throughout this paper we assume that $\Omega$ is an open, bounded subset with a Lipschitz boundary.

Let $\Omega_{1} \subset \Omega$ represent the region of interest (e.g., the foreground in image segmentation or the fluid region in Stokes flow), and let $\Omega_{2} = \Omega \setminus \Omega_{1}$ be the background. Since we are dealing with the geometry of interfaces, the natural mathematical framework is the space of functions of bounded variation, denoted by $BV(\Omega)$. We represent the partition using the characteristic function $u = \mathbbm{1}_{\Omega_1}$. Consequently, we define the admissible set for our binary partition variable as follows.
\begin{equation} \label{eq:admissible_set}
     BV(\Omega; \{0,1\}) := \left\{ u \in BV(\Omega) : u(\textbf{x}) \in \{0, 1\} \right\}.
\end{equation}

Our goal is to find the optimal function $u \in BV(\Omega; \{0,1\})$ and the associated parameters $\mathbf{c}$ by minimizing an energy functional. This functional is composed of a fidelity term, which measures the discrepancy between the data and the model parameters within each region, and a geometric regularization term that penalizes the complexity of the interface. We formulate the general energy functional as follows:
\begin{equation} \label{eq:general_energy}
    E(u, \mathbf{c}) = \int_{\Omega} u(\textbf{x}) F_{1}(\textbf{x}, \mathbf{c}) \, d\textbf{x} + \int_{\Omega} (1-u(\textbf{x})) F_{2}(\textbf{x}, \mathbf{c}) \, d\textbf{x} + \lambda \text{Per}(\Omega_{1}).
\end{equation}
Here, $\mathbf{c}$ represents a vector of parameters (e.g., the average intensity in image segmentation \cite{Chan_2001} or the velocity fields in Stokes flow \cite{Borrvall_2002}) that, together with the function $u$, completely parameterize the energy functional. The fidelity functions $F_1, F_2$ measure the cost in each region. We assume that for any fixed $\mathbf{c}$, $F_1(\cdot, \mathbf{c})$ and $F_2(\cdot, \mathbf{c})$ belong to $L^\infty(\Omega)$, ensuring that the data fidelity terms are finite. The term $\text{Per}(\Omega_{1})$ denotes the perimeter of the interface and, weighted by a parameter $\lambda > 0$, penalizes the complexity of the interface to ensure smoothness. 

\section{Derivation of the Algorithm}
\label{sec:algorithm}

In this section, we present our efficient weighted quantile filter scheme to minimize the energy functional defined in \eqref{eq:general_energy}. 

\subsection{Approximation of the Perimeter via Heat Kernel}
\label{subsec:approximation}

Direct minimization of the energy functional in \eqref{eq:general_energy} is computationally challenging due to its non-differentiability. To overcome this, we adopt an approximation of the perimeter functional based on the heat kernel. 

According to \cite{miranda2007short}, we can approximate the boundary perimeter using the short-time heat content energy:
\begin{equation} \label{eq:heat_approx}
    \text{Per}(\Omega_1) \approx \mathcal{P}_\tau(u) := \sqrt{\frac{\pi}{\tau}} \int_\Omega (1-u)(\textbf{x}) \left(G_\tau * u\right)(\textbf{x}) \, d\textbf{x},
\end{equation}
where $\tau > 0$ is a small time-scale parameter, and $G_\tau$ is the Gaussian kernel (heat kernel) with variance $2\tau$:
\begin{equation}
    G_\tau(\textbf{x}) = \frac{1}{4\pi\tau} \exp\left( -\frac{|\textbf{x}|^2}{4\tau} \right).
\end{equation}
As $\tau \to 0$, this energy $\mathcal{P}_\tau(u)$ converges to the perimeter of $\Omega_1$ in the sense of $\Gamma$-convergence. Substituting \eqref{eq:heat_approx} into \eqref{eq:general_energy}, we obtain the approximated energy functional $E_\tau$. 
\begin{equation} \label{eq:total_approx_energy}
    E_\tau(u, \textbf{c}) = \int_{\Omega} u(\textbf{x}) F_{1}(\textbf{x}, \mathbf{c}) + (1-u(\textbf{x})) F_{2}(\textbf{x}, \mathbf{c}) \, d\textbf{x} + \lambda \sqrt{\frac{\pi}{\tau}} \int_\Omega (1-u)(\textbf{x}) \left(G_\tau * u\right)(\textbf{x}) \, d\textbf{x}.
\end{equation}
This transformation allows us to handle the geometric regularization effectively using algebraic operations involving convolutions.

\subsection{Threshold Dynamics for Characteristic Functions}
\label{subsec:threshold_dynamics}

To minimize the energy functional \eqref{eq:total_approx_energy}, we adopt a coordinate descent strategy, following the alternating iterative framework utilized in \cite{Wang_2022b}. This iterative process involves two main steps:
\begin{enumerate}
    \item \textbf{Update the parameters:} Fix the partition $u$ and update the region parameters $\mathbf{c}$ (e.g., computing the mean intensity). This step is standard and minimizes the fidelity term.
    \item \textbf{Update the characteristic function:} Fix the parameters $\mathbf{c}$ and find the optimal partition $u$. This is the critical step where the geometric evolution occurs.
\end{enumerate}
In the following derivation, we focus on the second step: minimizing $E_\tau$ with respect to $u$ while keeping $\mathbf{c}$ fixed.

To minimize $E_\tau$ with respect to $u$, we employ a linearization technique. Since the term $\mathcal{P}_\tau$ is concave, we linearize the convolution term at the current iteration $u^k$. The variation of the energy with respect to $u$ leads to the following pointwise condition for $u^{k+1}$. We set $u^{k+1}(\textbf{x}) = 1$ if the functional derivative is non-positive, and $0$ otherwise. Specifically, $u^{k+1}(\textbf{x}) = 1$ if and only if:
\begin{equation}
    F_1(\textbf{x}) - F_2(\textbf{x}) + \lambda \sqrt{\frac{\pi}{\tau}} \left( 1 - 2 (G_\tau * u^k)(\textbf{x}) \right) \leq 0.
\end{equation}
Rearranging the terms, we obtain the threshold dynamics updating rule:
\begin{equation} \label{eq:threshold_update}
    u^{k+1}(\textbf{x}) = \mathbbm{1}_{ \left\{ (G_\tau * u^k)(\textbf{x}) \geq \frac{1}{2} + \frac{\sqrt{\tau}}{2\lambda\sqrt{\pi}} \left( F_1(\textbf{x}) - F_2(\textbf{x}) \right) \right\} }(\textbf{x}).
\end{equation}

\subsection{Weighted Quantile Filter for Level Set Functions}
\label{subsec:median_filter}

While the threshold dynamics scheme derived in Section \ref{subsec:threshold_dynamics} effectively minimizes the energy for binary functions, achieving sub-grid accuracy often requires a level-set formulation. In this subsection, we extend the binary scheme to continuous functions following the framework in \cite{Esedo_lu_2024, Guo_2025}.

Before detailing the specific update rules of the operator, it is crucial to establish the mathematical justification and the practical motivation for extending the binary framework to a continuous level-set formulation. As we will rigorously prove later in Section \ref{subsec:equivalence}, the original binary energy formulation \eqref{eq:total_approx_energy} is theoretically equivalent to a relaxed, continuous level-set energy functional \eqref{eq:relaxed_energy_def} via the coarea formula. 

The primary advantage of shifting to this continuous level-set formulation is twofold. First, from a numerical perspective, allowing the function $\phi$ to take continuous values in $[0, 1]$ effectively provides sub-grid resolution. This enables the interface to undergo infinitesimal spatial shifts, successfully overcoming the well-known pinning effect that severely limits traditional binary threshold dynamics. Second, from an optimization perspective, the continuous formulation convexifies the original combinatorial problem, paving the way for unconditionally stable and efficient iterative solvers while naturally recovering the sharp binary interfaces. With this continuous energy landscape in mind, the subsequent weighted quantile filter scheme operates as an exact minimizer of this relaxed formulation.

Let $\mathcal{S}_{\tau}$ denote the operator defined by the update rule \eqref{eq:threshold_update} for a binary characteristic function $u$. A crucial observation is that $\mathcal{S}_{\tau}$ satisfies the monotonicity property:

\begin{lemma}[Monotonicity] 
\label{lem:monotonicity}
    Let $u, v \in BV(\Omega; \{0,1\})$ be two binary functions such that $u(\textbf{x}) \le v(\textbf{x})$ for almost every $\textbf{x} \in \Omega$. Then, the operator $\mathcal{S}_{\tau}$ preserves this order, i.e.,
    \begin{equation}
        (\mathcal{S}_{\tau}u)(\textbf{x}) \le (\mathcal{S}_{\tau}v)(\textbf{x}), \quad \text{for all } \textbf{x} \in \Omega.
    \end{equation}
\end{lemma}

\begin{proof}
    Since the heat kernel $G_\tau$ is non-negative, the convolution operation is order-preserving. The assumption $u(\textbf{x}) \le v(\textbf{x})$ for almost every $\textbf{x} \in \Omega$ implies:
    \begin{equation}
        (G_\tau * u)(\textbf{x}) = \int_\Omega G_\tau(\textbf{x}-\textbf{y})u(\textbf{y}) \, d\textbf{y} \le \int_\Omega G_\tau(\textbf{x}-\textbf{y})v(\textbf{y}) \, d\textbf{y} = (G_\tau * v)(\textbf{x}).
    \end{equation}
    Let $T(\textbf{x}) = \frac{1}{2} + \frac{\sqrt{\tau}}{2\lambda\sqrt{\pi}} \left( F_1(\textbf{x}) - F_2(\textbf{x}) \right)$ be the local threshold. If $(\mathcal{S}_{\tau}u)(\textbf{x}) = 1$, it implies $(G_\tau * u)(\textbf{x}) \geq T(\textbf{x})$. By the inequality above, we necessarily have $(G_\tau * v)(\textbf{x}) \geq T(\textbf{x})$, which implies $(\mathcal{S}_{\tau}v)(\textbf{x}) = 1$. Since the functions are binary, this confirms that $\mathcal{S}_{\tau}u \le \mathcal{S}_{\tau}v$.
\end{proof}

This monotonicity allows us to apply the scheme to a general continuous level-set function $\phi\in[0, 1]$ by applying the operator to all its super-level sets simultaneously. For a level set function $\phi^k$, its super-level set at level $\mu$ is defined as $\mathcal{T}_{\mu}\phi^k = \{\textbf{x} \in \Omega : \phi^k(\textbf{x}) \ge \mu\}$. The updated value $\phi^{k+1}(\textbf{x})$ is determined by the highest level $\mu$ for which $\textbf{x}$ remains inside the set evolved from $\mathcal{T}_{\mu}\phi^k$. 

Mathematically, this leads to the following explicit update formula involving the heat kernel weights:
\begin{equation} \label{eq:median_update}
    \phi^{k+1}(\textbf{x}) = \sup \left\{ \{0\} \cup \left\{ \mu \in (0,1] : \int_{\{\textbf{y} : \phi^k(\textbf{y}) \ge \mu\}} G_\tau(\textbf{x}-\textbf{y}) \, d\textbf{y} \geq T(\textbf{x}) \right\} \right\}.
\end{equation}
Here, $T(\textbf{x}) = \frac{1}{2} + \frac{\sqrt{\tau}}{2\lambda\sqrt{\pi}} \left( F_1(\textbf{x}) - F_2(\textbf{x}) \right)$ is the local threshold term derived in Section \ref{subsec:threshold_dynamics}. 

The formula \eqref{eq:median_update} essentially searches for a weighted quantile of the distribution of $\phi^k$ in the neighborhood of $\textbf{x}$. Specifically, if we treat $G_\tau(\textbf{x}-\cdot)$ as a probability measure, $\phi^{k+1}(\textbf{x})$ is the value $\mu$ such that the cumulative weight of neighbors with values greater than or equal to $\mu$ exceeds the threshold $T(\textbf{x})$.

\subsection{Algorithm Description}
\label{subsec:algorithm}

Based on the weighted quantile interpretation in \eqref{eq:median_update}, the numerical implementation avoids explicit convolution and instead relies on a local sorting strategy. This corresponds to the ``Weighted Quantile Filter'' solver.

In the numerical implementation, to enhance computational efficiency, we employ a discretization approximation for the convolution step. Specifically, we restrict the support of the convolution kernel to a discrete set of points equidistant from the center $\textbf{x}$ (i.e., on a circle of fixed radius $R = \sqrt{2\tau}$). This specific choice of sampling on a 1D circle rather than within a full 2D disk is driven by a critical balance between geometric accuracy and computational efficiency. Mathematically, by Taylor expansion, the uniform median on a circle of radius $R = \sqrt{2\tau}$ yields the exact same leading-order approximation to the full 2D Gaussian weighted median over a disk \cite{Oberman_2004}. Algorithmically, the quantile filter scheme requires sorting the sampled values. Sampling on a circle reduces the number of points from an area complexity $O(M^2)$ to a boundary complexity $O(M)$, significantly accelerating the sorting process. While using a full disk might marginally reduce higher-order truncation errors, it does not alter the leading-order geometric evolution but would reduce the computational speed.

Due to the radial symmetry of the heat kernel $G_\tau$ defined in Section \ref{subsec:approximation}, its value remains constant along the circle. Consequently, the discrete weights $w_j$ at these sample points are identical, effectively reducing the convolution to a uniform average on the circle. Although this geometric approximation might alter the specific pre-factor derived in Section \ref{subsec:threshold_dynamics}, in practice, this constant factor is absorbed into the tunable weight parameter $\lambda$ and does not affect the variational nature of the model.

\begin{algorithm}[H]
\caption{Weighted Quantile Filter Scheme}
\label{alg:median_filter}
\begin{algorithmic}[1]
\STATE \textbf{Input:} Data $F_1, F_2$, parameters $\lambda, \tau$.
\STATE \textbf{Initialization:} Choose a continuous level-set function $\phi^0$.
\FOR{$k = 0$ to $K_{max}$}
    \STATE \textbf{Step 1: Update Parameters.}
    \STATE Compute region parameters $\mathbf{c}^k$ based on the current level-set function $\phi^k$.
    
    \STATE \textbf{Step 2: Quantile Filtering.}
    \STATE For each grid point $\textbf{x}$:
    \STATE \quad (a) Calculate the threshold $T(\textbf{x}) = \frac{1}{2} + \frac{\sqrt{\tau}}{2\lambda\sqrt{\pi}} \left( F_1(\textbf{x}) - F_2(\textbf{x}) \right)$.
    \STATE \quad (b) Update $\phi^{k+1}(\textbf{x}) = \text{GeneralQuantileFilter}(\phi^k, T(\textbf{x}))$.
    
    \STATE \textbf{Step 3: Check Convergence.}
    \STATE If $\|\phi^{k+1} - \phi^k\| < \epsilon$, break.
\ENDFOR
\STATE \textbf{Output:} Final level-set function $\phi^{k+1}$.
\end{algorithmic}
\end{algorithm}

\begin{remark}
    In Step 2(b), the operator $\text{GeneralQuantileFilter}(\phi, T)$ represents a class of solvers that evolve the interface according to the mean curvature flow subject to an external force $T$. In Section \ref{sec:numerical_experiments}, we will detail the specific implementations of this operator.    
\end{remark}

\section{Theoretical Analysis}
\label{sec:theory}

In this section, we establish the theoretical foundations of the proposed algorithm. We first prove the equivalence between the level-set formulation and the original binary characteristic formulation using the coarea formula. Furthermore, we prove the existence and uniqueness of the solution and justify its binary property. Next, we demonstrate the unconditional energy stability of the iterative scheme. Finally, we discuss the consistency of the scheme with the mean curvature flow and analyze its stability properties based on the monotonicity principle established in Lemma \ref{lem:monotonicity}.

\subsection{Equivalence of Level-Set and Binary Formulations}
\label{subsec:equivalence}

In this subsection, we establish the theoretical equivalence between the binary optimization problem introduced in Section \ref{sec:model} and the continuous level-set formulation. To streamline the notation and highlight the role of tunable parameters, we introduce an effective weighting parameter $\tilde{\lambda}$ that absorbs the geometric scaling factors inherent in the heat kernel approximation. We define the effective parameter $\tilde{\lambda}$ as:
\begin{equation} \label{eq:effective_lambda}
    \tilde{\lambda} := \lambda \sqrt{\frac{\pi}{\tau}}.
\end{equation}

To construct a continuous relaxation of the binary energy functional \eqref{eq:total_approx_energy}, we rely on the layer cake representation following \cite{Esedo_lu_2024}. The core rationale is that a natural way to define the energy of a continuous level-set function $\phi \in BV(\Omega; [0,1])$ is to integrate the binary energy evaluated on its super-level sets, $\mathcal{T}_\mu\phi = \{\textbf{x} \in \Omega : \phi(\textbf{x}) \ge \mu\}$, across all threshold levels $\mu \in [0, 1]$. Using this threshold decomposition, we formally define the relaxed energy. 

\begin{definition}[Relaxed Energy Functional via Layer Cake Representation]
    \label{def:relaxed_energy_functional}
    For any function $\phi \in BV(\Omega; [0,1])$ and fixed parameter $\textbf{c}$, we define the relaxed energy functional $\mathcal{E}_\tau(\phi, \textbf{c})$ as:
    \begin{equation}
    \label{eq:layer_cake_total}
    \begin{split}
    \mathcal{E}_\tau(\phi, \textbf{c}) := {} & \frac{2}{\tilde{\lambda}} \int_{0}^{1} E_{\tau}(\mathbbm{1}_{\mathcal{T}_\mu\phi}, \textbf{c}) \, d\mu \\
    = {} & 
        2 \int_0^1 \left[ \int_\Omega (1 - \mathbbm{1}_{\mathcal{T}_\mu\phi})(G_\tau * \mathbbm{1}_{\mathcal{T}_\mu\phi}) \, d\textbf{x} + \frac{1}{\tilde{\lambda}} \int_{\Omega} \mathbbm{1}_{\mathcal{T}_\mu\phi} F_1 + (1-\mathbbm{1}_{\mathcal{T}_\mu\phi}) F_2 \, d\textbf{x} \right] d\mu.
    \end{split}
    \end{equation}
    By applying the generalized coarea identity \cite{Esedo_lu_2024} to the approximate perimeter term:
    \begin{equation} \label{eq:coarea}
        2 \int_0^1 \left[ \int_\Omega (1 - \mathbbm{1}_{\mathcal{T}_\mu\phi})(G_\tau * \mathbbm{1}_{\mathcal{T}_\mu\phi}) \, d\textbf{x} \right] d\mu
        = \int_\Omega\int_\Omega G_\tau(\textbf{x}-\textbf{y})|\phi(\textbf{x})-\phi(\textbf{y})| \, d\textbf{y} \, d\textbf{x}.
    \end{equation}
    By applying Cavalieri's Principle \cite{Lieb_2001} to the data fidelity terms:
    \begin{equation} \label{eq:layer_cake}
        2 \int_0^1 \left[ \frac{1}{\tilde{\lambda}} \int_{\Omega} \mathbbm{1}_{\mathcal{T}_\mu\phi} F_1 + (1-\mathbbm{1}_{\mathcal{T}_\mu\phi}) F_2 \, d\textbf{x} \right] d\mu
        = \frac{2}{\tilde{\lambda}} \int_{\Omega}\phi F_1+(1-\phi)F_2 \, d\textbf{x}.
    \end{equation}
    Therefore, the relaxed energy functional explicitly evaluates to the following double-integral form:
    \begin{equation}
    \label{eq:relaxed_energy_def}
    \begin{split}
    \mathcal{E}_\tau(\phi, \textbf{c}) = {} & 
        \int_\Omega\int_\Omega G_\tau(\textbf{x}-\textbf{y}) 
            |\phi(\textbf{x})-\phi(\textbf{y})| \, d\textbf{y} \, d\textbf{x} \\
    & + \frac{2}{\tilde{\lambda}} \int_{\Omega} 
        \Big[ \phi(\textbf{x})F_1(\textbf{x}, \textbf{c}) 
             + (1-\phi(\textbf{x}))F_2(\textbf{x}, \textbf{c}) 
        \Big] \, d\textbf{x}.
    \end{split}
    \end{equation}
\end{definition}

\begin{remark}
    It is worth noting the relationship between the relaxed formulation \eqref{eq:layer_cake_total} and the original discrete formulation \eqref{eq:total_approx_energy} when evaluated at binary states. If the argument $\phi$ is strictly a binary characteristic function, i.e., $\phi \in BV(\Omega; \{0,1\})$, its super-level set indicator trivially satisfies $\mathbbm{1}_{\mathcal{T}_{\mu}\phi} = \phi$ for all threshold levels $\mu \in [0, 1]$. Consequently, the integral in Definition \ref{def:relaxed_energy_functional} simplifies directly to $\mathcal{E}_{\tau}(\phi, \textbf{c}) = \frac{2}{\tilde{\lambda}} E_{\tau}(\phi, \textbf{c})$. This straightforward consistency confirms that the continuous functional $\mathcal{E}_{\tau}$ acts as an exact relaxation of the discrete functional up to the global scaling constant $\frac{2}{\tilde{\lambda}}$.
\end{remark}

Inspired by \cite{Chan_2006}, we present the following equivalence theorem for our model.

\begin{theorem}[Equivalence via Threshold Decomposition]
    \label{thm:equivalence}
    Let $\phi^*$ be a global minimizer of the relaxed energy functional $\mathcal{E}_\tau(\phi, \textbf{c})$. Then, for almost every threshold level $\mu \in [0, 1]$, the characteristic function $\mathbbm{1}_{\mathcal{T}_\mu\phi^*}$ associated with the super-level set 
    \begin{equation}
        \mathcal{T}_\mu\phi^* = \{\textbf{x} \in \Omega : \phi^*(\textbf{x}) \ge \mu\}
    \end{equation}
    is a global minimizer of the original binary energy functional \eqref{eq:total_approx_energy}.
\end{theorem}

\begin{proof}
    By Definition \ref{def:relaxed_energy_functional}, the relaxed energy is exactly the normalized integration of the binary energy evaluated on its super-level sets:
    \[
    \mathcal{E}_{\tau}(\phi^{*}, \textbf{c}) = \frac{2}{\tilde{\lambda}} \int_{0}^{1} E_{\tau}\left(\mathbbm{1}_{\mathcal{T}_{\mu}\phi^{*}}, \textbf{c}\right) \, d\mu .
    \]
    Let 
    \[
    m := \inf_{u \in BV(\Omega; \{0,1\})} E_{\tau}(u, \textbf{c})
    \]
    be the global minimum of the binary problem. Since $\mathbbm{1}_{\mathcal{T}_{\mu}\phi^{*}}$ is a binary function, we necessarily have $E_{\tau}(\mathbbm{1}_{\mathcal{T}_{\mu}\phi^{*}}, \textbf{c}) \ge m$ for all $\mu \in [0, 1]$. Hence, 
    \[
    \mathcal{E}_{\tau}(\phi^{*}, \textbf{c}) \ge \frac{2}{\tilde{\lambda}} m .
    \]
    On the other hand, since $\phi^{*}$ is a global minimizer of $\mathcal{E}_{\tau}$, for any binary function $u$ we have $\mathcal{E}_{\tau}(\phi^{*}, \textbf{c}) \le \mathcal{E}_{\tau}(u, \textbf{c}) = \frac{2}{\tilde{\lambda}} E_{\tau}(u, \textbf{c})$. Taking the infimum over all binary $u$ gives
    \[
    \mathcal{E}_{\tau}(\phi^{*}, \textbf{c}) \le \frac{2}{\tilde{\lambda}} m .
    \]
    Therefore, 
    \[
    \mathcal{E}_{\tau}(\phi^{*}, \textbf{c}) = \frac{2}{\tilde{\lambda}} m .
    \]
    
    Now suppose, for contradiction, that there exists a set of positive measure in $\mu \in [0,1]$ such that $E_{\tau}(\mathbbm{1}_{\mathcal{T}_{\mu}\phi^{*}}, \textbf{c}) > m$, it would strictly imply $\mathcal{E}_{\tau}(\phi^{*}, \textbf{c}) > \frac{2}{\tilde{\lambda}} m$. Therefore,
    \[
    E_{\tau}(\mathbbm{1}_{\mathcal{T}_{\mu}\phi^{*}}, \textbf{c}) = m
    \]
    must hold for almost every $\mu \in [0, 1]$, proving that the super-level sets are indeed global minimizers of the binary energy.
\end{proof}

\begin{remark}
    In Section \ref{subsec:equivalence}, the parameter \textbf{c} is considered fixed. For notational simplicity, we will omit \textbf{c}.
\end{remark}

While Theorem \ref{thm:equivalence} ensures that the super level sets of any minimizer are optimal for the binary problem, it does not explicitly guarantee that the minimizer $\phi^*$ itself is binary. In the following, we demonstrate that for generic data, the convex relaxation inherently enforces the binary constraint $\phi \in \{0,1\}$. This justifies the use of the convex model without requiring explicit penalization terms.

To formalize this, we consider the space $X = L^1(\Omega)$ and the admissible set:
\begin{equation}
    \mathcal{K} = \left\{ \phi \in BV(\Omega) : 0 \le \phi(\textbf{x}) \le 1 \text{ a.e., and } |D\phi|(\Omega) \le C \right\},
\end{equation}
where $|D\phi|(\Omega)$ is called the variation of $\phi$ in $\Omega$. Note that the set $\mathcal{K}$ is uniformly bounded in the $BV(\Omega)$ norm. Since the embedding of $BV(\Omega)$ into $L^1(\Omega)$ is compact \cite{Ambrosio_2000}, any bounded subset of $BV(\Omega)$ is relatively compact in $L^1(\Omega)$. Therefore, $\mathcal{K}$ is relatively compact in $L^1(\Omega)$. Furthermore, since $\mathcal{K}$ is closed in $L^1$, it becomes a compact subset of $X$.

Define the interaction functional $J(\phi) = \int_\Omega\int_\Omega G_\tau(\textbf{x}-\textbf{y})|\phi(\textbf{x})-\phi(\textbf{y})| \, d\textbf{y} \, d\textbf{x}$. Note that $J$ is convex thus lower semicontinuous. The total energy is, up to an additive constant independent of $\phi$, a linear perturbation of $J$ by the data field $h(\textbf{x}) = \frac{2}{\tilde{\lambda}}(F_1(\textbf{x}) - F_2(\textbf{x}))$:
\begin{equation}
\mathcal{E}_{\tau}(\phi) = J(\phi) + \int_{\Omega} \phi(\textbf{x}) h(\textbf{x}) \, d\textbf{x}.
\end{equation}

\begin{lemma}[Existence and Uniqueness] \label{lem:existence_uniqueness}
    Let $X$ be a Banach space, and let $\mathcal{K} \subset X$ be a compact, convex subset. 
    Let $J: \mathcal{K} \to \mathbbm{R}$ be a lower-bounded lower semicontinuous functional. 
    Consider the perturbed functional $\mathcal{E}_{\tau}(\phi) = J(\phi) + \langle h, \phi \rangle$ defined on $\mathcal{K}$. 
    Then there exists a dense $G_\delta$ subset $\mathcal{G}$ of $X^*$ such that for every perturbation $h \in \mathcal{G}$, the functional $\mathcal{E}_{\tau}$ admits a unique global minimizer $\phi^*$ in $\mathcal{K}$.
\end{lemma}

\begin{proof}
    See Appendix \ref{app:existence_uniqueness}.
\end{proof}

\begin{theorem}[Generic Binary Enforcement]
    \label{thm:binary_enforcement}
    Let $\mathcal{K}$ be the compact convex set defined above, and let $\mathcal{E}_{\tau}(\phi) = J(\phi) + \int_{\Omega} \phi(\mathbf{x}) h(\mathbf{x}) \, d\mathbf{x}$ be the energy functional defined on $\mathcal{K}$. Then there exists a dense $G_\delta$ subset $\mathcal{G}$ of $L^\infty(\Omega)$ such that for every $h \in \mathcal{G}$, the functional $\mathcal{E}_{\tau}$ has a unique minimizer $\phi^*$ in $\mathcal{K}$. Furthermore, if $\mathbbm{1}_{\mathcal{T}_\mu\phi^*} \in \mathcal{K}$ for almost every $\mu \in (0,1)$, then $\phi^*(\mathbf{x}) \in \{0,1\}$ almost everywhere.
\end{theorem}

\begin{proof}
    By Lemma \ref{lem:existence_uniqueness}, there exists a dense $G_\delta$ subset $\mathcal{G}$ of $L^\infty(\Omega)$ such that for every $h \in \mathcal{G}$, the functional $\mathcal{E}_{\tau}$ has a unique minimizer $\phi^*$ in $\mathcal{K}$. This establishes the uniqueness statement. It remains to show that $\phi^*$ is binary. To simplify the notation, let $\tilde{\mathcal{E}}_{\tau} = \frac{\tilde{\lambda}}{2} \mathcal{E}_{\tau}$, which clearly shares the same unique minimizer $\phi^*$.
    
    Suppose, for contradiction, that $\phi^*$ is not binary almost everywhere. Then the set $A = \{x \in \Omega: 0 < \phi^*(x) < 1\}$ has positive measure. From \eqref{eq:layer_cake_total}, the normalized relaxed energy of $\phi^*$ can be expressed as the integral of the binary energy of its super-level sets:
    $$ \tilde{\mathcal{E}}_{\tau}(\phi^*) = \int_0^1 E_{\tau}(\mathbbm{1}_{\mathcal{T}_\mu\phi^*}) d\mu. $$
    
    Let $m = \inf_{u \in BV(\Omega; \{0,1\} ) \cap \mathcal{K}} E_{\tau}(u)$. Since the super-level set indicator $\mathbbm{1}_{\mathcal{T}_\mu\phi^*}$ is a binary function in $\mathcal{K}$, we have $E_{\tau}(\mathbbm{1}_{\mathcal{T}_\mu\phi^*}) \ge m$ for all $\mu \in (0,1)$. Integrating this inequality yields $\tilde{\mathcal{E}}_{\tau}(\phi^*) \ge m$.
    
    Conversely, for any binary function $v \in BV(\Omega; \{0,1\}) \cap \mathcal{K}$, the normalized relaxed energy coincides with the binary energy, meaning $\tilde{\mathcal{E}}_{\tau}(v) = E_{\tau}(v)$. Since $\phi^*$ is the global minimizer of $\tilde{\mathcal{E}}_{\tau}$ over the entire relaxed set $\mathcal{K}$, it must hold that $\tilde{\mathcal{E}}_{\tau}(\phi^*) \le \tilde{\mathcal{E}}_{\tau}(v) = E_{\tau}(v)$ for all binary functions $v$. Taking the infimum over all such $v$ yields $\tilde{\mathcal{E}}_{\tau}(\phi^*) \le m$.
    
    Combining these inequalities, we conclude that $\tilde{\mathcal{E}}_{\tau}(\phi^*) = m$. For the integral of $E_{\tau}(\mathbbm{1}_{\mathcal{T}_\mu\phi^*})$ to exactly equal $m$ while the integrand is bounded below by $m$, it must be that $E_{\tau}(\mathbbm{1}_{\mathcal{T}_\mu\phi^*}) = m$ for almost every $\mu \in (0,1)$.
    
    Now, fix a $\mu \in (0,1)$ such that the characteristic function $\psi_\mu = \mathbbm{1}_{\mathcal{T}_\mu\phi^*}$ achieves this minimum, meaning $E_{\tau}(\psi_\mu) = m$. Because $\psi_\mu$ is a binary function, its normalized relaxed energy is exactly its binary energy: 
    $$ \tilde{\mathcal{E}}_{\tau}(\psi_\mu) = E_{\tau}(\psi_\mu) = m = \tilde{\mathcal{E}}_{\tau}(\phi^*). $$
    
    Thus, $\psi_\mu$ attains the same global minimum value as $\phi^*$, making $\psi_\mu$ also a global minimizer of the normalized relaxed energy $\tilde{\mathcal{E}}_{\tau}$ as well as the original functional $\mathcal{E}_{\tau}$. However, since $\phi^*$ takes values strictly between $0$ and $1$ on the set $A$ of positive measure, while $\psi_\mu$ takes only values $0$ or $1$, we have $\psi_\mu \neq \phi^*$. This yields two distinct global minimizers for $\mathcal{E}_{\tau}$, which directly contradicts the uniqueness guaranteed by the uniqueness statement. Therefore, $\phi^*$ must be binary almost everywhere.
\end{proof}

\begin{remark}
    The conditional assumption $\mathbbm{1}_{\mathcal{T}_\mu\phi^*} \in \mathcal{K}$, i.e., $|D(\mathbbm{1}_{\mathcal{T}_\mu\phi^*})|(\Omega) \le C$, is well-justified in practical interface optimal design problems. For a physical system where the true optimal binary partition has a finite perimeter, one can easily choose the regularization parameter $C$ to be sufficiently large. In this regime, the bounded variation constraint acts strictly as an inactive constraint near the physical solution, naturally guaranteeing that the level-set perimeters remain strictly bounded by $C$.
\end{remark}

\subsection{Derivation of the Scheme and Energy Stability}
\label{subsec:scheme_derivation}

In this subsection, we derive the explicit update rule of our quantile filter algorithm and strictly prove its unconditional energy stability. The primary challenge lies in the non-local interaction term involving $|\phi(\textbf{x})-\phi(\textbf{y})|$, which couples the values of the level-set function across the domain. To decouple these interactions and derive an efficient pointwise solver, we adopt the movement limiter approach proposed in \cite{Esedo_lu_2024}.

\subsubsection{Derivation of the Iterative Scheme}
\label{subsubsec:derivation}

The non-local interaction term in $\mathcal{E}_\tau$ couples $\phi(\mathbf{x})$ and $\phi(\mathbf{y})$, making direct minimization computationally expensive. To decouple these interactions and derive an efficient pointwise solver, we introduce an auxiliary functional $\mathcal{M}_k(\phi)$ proposed in \cite{Esedo_lu_2024}, referred to as the movement limiter. The motivation of $\mathcal{M}_k(\phi)$ is to exactly cancel the computationally expensive coupled term $|\phi(\mathbf{x}) - \phi(\mathbf{y})|$ using the known previous state $\phi^k$. Furthermore, $\mathcal{M}_k(\phi) \ge 0$ and $\mathcal{M}_k(\phi^k) = 0$, therefore it prevents drastic changes in a single step, which rigorously ensures monotonic energy decay.

With this rationale, we define the movement limiter formally as follows:

\begin{definition}[Movement Limiter]
    \label{def:movement_limiter}
    Let $\phi^k$ be the solution in iteration $k$. Define the movement limiter $\mathcal{M}_k(\phi)$ as:
    \begin{equation}
        \label{eq:movement_limiter_def}
        \mathcal{M}_k(\phi) = \int_\Omega \int_\Omega G_\tau(\mathbf{x} - \mathbf{y}) \left( 2|\phi(\mathbf{x}) - \phi^k(\mathbf{y})| - |\phi(\mathbf{x}) - \phi(\mathbf{y})| - |\phi^k(\mathbf{x}) - \phi^k(\mathbf{y})| \right) d\mathbf{y} d\mathbf{x}.
    \end{equation}
\end{definition}

Instead of directly minimizing $\mathcal{E}_\tau(\phi, \textbf{c})$, we minimize alternative energy $\mathcal{E}_\tau(\phi, \textbf{c}) + \mathcal{M}_k(\phi)$. The following theorem shows that the global minimizer of this alternative energy is exactly the quantile filter scheme presented in Algorithm \ref{alg:median_filter}.

\begin{theorem}[Exact Iterative Scheme]
    The global minimizer of the alternative problem
    \begin{equation} \label{eq:alternative_min}
        \phi^{k+1} = \arg\min_{\phi \in BV(\Omega; [0,1])} \left( \mathcal{E}_\tau(\phi, \textbf{c}) + \mathcal{M}_k(\phi) \right)
    \end{equation}
    is given explicitly by the weighted quantile formula \eqref{eq:median_update}.
\end{theorem}

\begin{proof}
    First, we simplify the alternative energy. Substituting \eqref{eq:relaxed_energy_def} and \eqref{eq:movement_limiter_def} into \eqref{eq:alternative_min}, the coupled terms cancel exactly. Discarding terms independent of $\phi$ (which do not affect the minimization), the problem reduces to:
    \begin{equation}
    \begin{split}
    \phi^{k+1} = \arg\min_{\phi \in BV(\Omega; [0,1])} \biggl[ & 2 \int_\Omega\int_\Omega G_\tau(\textbf{x}-\textbf{y})|\phi(\textbf{x})-\phi^k(\textbf{y})| \, d\textbf{y} \, d\textbf{x} \\
    & + \frac{2}{\tilde{\lambda}} \int_{\Omega}\left(\phi F_1 + (1-\phi) F_2\right) \, d\textbf{x} \biggr].
    \end{split}
    \end{equation}
    Dividing by 2, we observe that the integrand does not contain any spatial derivatives of $\phi$. Thus, the global minimization decouples into independent pointwise minimization problems. For each $\textbf{x} \in \Omega$, we find the value $\xi = \phi^{k+1}(\textbf{x}) \in [0,1]$ that minimizes the local potential $J_\textbf{x}(\xi)$:
    \begin{equation}
        J_\textbf{x}(\xi) = \int_\Omega G_\tau(\textbf{x}-\textbf{y})|\xi-\phi^k(\textbf{y})| \, d\textbf{y} + \frac{1}{\tilde{\lambda}} \left( \xi F_1(\textbf{x}) + (1-\xi) F_2(\textbf{x}) \right).
    \end{equation}
    Since $J_\textbf{x}(\xi)$ is a convex function of $\xi$, the optimal value is determined by the first-order optimality condition. We first evaluate the weak derivative of $J_\textbf{x}$ with respect to $\xi$:
    \begin{equation}
        \frac{dJ_\textbf{x}}{d\xi} = \int_\Omega G_\tau(\textbf{x}-\textbf{y}) \cdot \text{sgn}(\xi - \phi^k(\textbf{y})) \, d\textbf{y} + \frac{1}{\tilde{\lambda}}(F_1(\textbf{x}) - F_2(\textbf{x})).
    \end{equation}
    Using the normalization condition for $G_\tau$ and decomposing the sign function integral into regions where $\phi^k < \xi$ and $\phi^k > \xi$, we have:
    \begin{equation}
         \int_\Omega G_\tau(\textbf{x}-\textbf{y}) \cdot \text{sgn}(\xi - \phi^k(\textbf{y})) \, d\textbf{y} = 1 - 2 \int_{\{\textbf{y}:\phi^k(\textbf{y}) > \xi\}} G_\tau(\textbf{x}-\textbf{y}) \, d\textbf{y}.
    \end{equation}
    Substituting this back into the weak derivative of $J_\textbf{x}$:
    \begin{equation}
        \frac{dJ_\textbf{x}}{d\xi} = 1 - 2 \int_{\{\textbf{y}:\phi^k(\textbf{y}) > \xi\}} G_\tau(\textbf{x}-\textbf{y}) \, d\textbf{y} + \frac{1}{\tilde{\lambda}}(F_1 - F_2).
    \end{equation}
    Denoting that $S(\xi)=\int_{\{\textbf{y}:\phi^k(\textbf{y}) > \xi\}} G_\tau(\textbf{x}-\textbf{y}) \, d\textbf{y}$ and recalling that $T(\textbf{x}) = \frac{1}{2} + \frac{1}{2\tilde{\lambda}} \left( F_1(\textbf{x}) - F_2(\textbf{x}) \right)$, we can write the weak derivative more compactly as:
    \begin{equation}
        \frac{dJ_\textbf{x}}{d\xi} = 2(T(\textbf{x})-S(\xi)).
    \end{equation}
    Because the super-level set shrinks as $\xi$ increases, $S(\xi)$ is a monotonically non-increasing function. Consequently, the weak derivative $\frac{dJ_\textbf{x}}{d\xi}$ is a monotonically non-decreasing function, which governs the minimization process under two distinct scenarios:
    \begin{enumerate}
        \item If $\frac{dJ_\textbf{x}}{d\xi} > 0$ for all $\xi \in [0,1]$ (i.e., $T(\mathbf{x}) > S(0)$), the function $J_{\mathbf{x}}(\xi)$ is strictly increasing on the entire domain. The constrained global minimum is naturally achieved at the lower boundary, $\xi^* = 0$.
    
        \item If $J_{\mathbf{x}}(\xi)$ is not strictly increasing, there exists a non-empty valid set of $\xi$ where the weak derivative is non-positive (i.e., $\frac{dJ_\textbf{x}}{d\xi} \le 0$). Because the weak derivative is non-decreasing, the convex function $J_{\mathbf{x}}(\xi)$ initially decreases and subsequently increases. Its global minimum is therefore exactly attained at the supremum of the set where the weak derivative remains non-positive.
    \end{enumerate}
    Setting $\frac{dJ_\textbf{x}}{d\xi} \le 0$ yields the condition $S(\xi) \ge T(\mathbf{x})$. By unifying both scenarios, the optimal solution $\phi^{k+1}(\mathbf{x}) = \xi^*$ is explicitly given by the weighted quantile filter formula \eqref{eq:median_update}.
\end{proof}

\subsubsection{Total Energy Stability}

We now prove that the proposed algorithm, which alternates between updating the parameters $\textbf{c}$ and the level-set function $\phi$, guarantees the monotonic decay of the total energy.

\begin{theorem}[Unconditional Stability]
    \label{thm:unconditional_stability}
    Let $\{(\phi^k, \textbf{c}^k)\}$ be the sequence generated by the alternating iterative algorithm. Then, the total energy is non-increasing:
    \begin{equation}
        \mathcal{E}_\tau(\phi^{k+1}, \textbf{c}^{k+1}) \le \mathcal{E}_\tau(\phi^k, \textbf{c}^k).
    \end{equation}
\end{theorem}

\begin{proof}
    First, the parameters are updated to minimize the energy with fixed $\phi^k$:
    $$ \textbf{c}^{k+1} = \arg\min_\textbf{c} \mathcal{E}(\phi^k, \textbf{c}). $$
    By definition, this implies:
    \begin{equation} \label{eq:decay_c}
        \mathcal{E}_\tau(\phi^k, \textbf{c}^{k+1}) \le \mathcal{E}_\tau(\phi^k, \textbf{c}^k).
    \end{equation}
    Next, $\phi$ is updated using the scheme \eqref{eq:median_update}. Since $\phi^{k+1}$ minimizes the alternative energy $\mathcal{E}_\tau(\cdot, \textbf{c}^{k+1}) + \mathcal{M}_k(\cdot)$, we have:
    \begin{align*}
        \mathcal{E}_\tau(\phi^{k+1}, \textbf{c}^{k+1}) &\le \mathcal{E}_\tau(\phi^{k+1}, \textbf{c}^{k+1}) + \mathcal{M}_k(\phi^{k+1}) & \text{(Non-negativity of } \mathcal{M}_k \text{)} \\
        &\le \mathcal{E}_\tau(\phi^k, \textbf{c}^{k+1}) + \mathcal{M}_k(\phi^k) & \text{(Minimality of } \phi^{k+1} \text{)} \\
        &= \mathcal{E}_\tau(\phi^k, \textbf{c}^{k+1}). & \text{(Contact property } \mathcal{M}_k(\phi^k)=0 \text{)}
    \end{align*}
    Combining the inequalities from both steps:
    \begin{equation}
        \mathcal{E}_\tau(\phi^{k+1}, \textbf{c}^{k+1}) \le \mathcal{E}_\tau(\phi^k, \textbf{c}^{k+1}) \le \mathcal{E}_\tau(\phi^k, \textbf{c}^k).
    \end{equation}
    Thus, the algorithm is unconditionally energy stable.
\end{proof}

\section{Numerical Experiments}
\label{sec:numerical_experiments}

In this section, we present several numerical experiments to validate the theoretical properties and practical performance of the proposed weighted-quantile-filter-based framework. Our algorithm is designed as a continuous threshold dynamics scheme: it effectively integrates the numerical stability of traditional threshold dynamics with the geometric continuity of level-set representations. This unique combination allows for robust interface evolution even under challenging numerical settings.

\subsection{Numerical Implementation via Weighted Quantile Filter}
\label{subsec:implementation}

The core of our numerical framework lies in the efficient evaluation of the convolution thresholding step. Unlike the traditional binary MBO scheme, which relies on direct convolution thresholding that leads to the pinning effect, our approach implements the update via a weighted quantile filter, extending the median filter scheme acting on the continuous level-set function $\phi$ \cite{Oberman_2004}.

For a discretized domain, the Gaussian kernel $G_{\tau}$ is approximated by a discrete mask $W$ with weights $\{w_j\}_{j=1}^M$. The update for the level-set function $\phi^{k+1}(\textbf{x})$ at each grid point $\textbf{x}$ is computed as follows:

\begin{enumerate}
    \item \textbf{Neighbor Evaluation:} We sample the values of the current level-set function $\phi^k$ within the neighborhood of $\textbf{x}$. To maintain sub-grid accuracy, these values $\{\phi^k(\textbf{x} + \textbf{y}_j)\}_{j=1}^M$ can be evaluated using bilinear interpolation.
    
    \item \textbf{Sorting:} The sampled values are sorted in descending order, denoted as $v_{(1)} \ge v_{(2)} \ge \dots \ge v_{(M)}$.
    
    \item \textbf{Quantile Selection:} We calculate the cumulative weights of the sorted neighbors. The new value $\phi^{k+1}(\textbf{x})$ is determined by selecting the value $v_{(m)}$ that corresponds to the specific quantile required by \eqref{eq:median_update}.
\end{enumerate}

This procedure, summarized in Algorithm \ref{alg:median_filter_continuous}, allows the level-set function to take values from a continuous range rather than being restricted to a binary set. By selecting the update value from the sorted continuous neighbors, the interface can move by arbitrary sub-grid distances, thereby effectively eliminating the pinning effect associated with spatial discretization.

\begin{algorithm}[H]
\caption{Continuous Weighted Quantile Filtering Scheme}
\label{alg:median_filter_continuous}
\begin{algorithmic}[1]
\STATE \textbf{Input:} Current level-set function $\phi^k$, time step $\tau$, kernel weights $\{w_j\}_{j=1}^M$.
\STATE \textbf{Initialization:} Generate the discrete mask $W = \{w_j\}$ approximating $G_{\tau}$.

\FOR{each grid point \textbf{x} in the domain}
    \STATE Sample neighbors $\{\phi^k(\textbf{x} + \textbf{y}_j)\}_{j=1}^M$ using bilinear interpolation.
    \STATE Sort the sampled values in descending order: $v_{(1)} \ge v_{(2)} \ge \dots \ge v_{(M)}$.
    \STATE Compute cumulative weights $S_m = \sum_{l=1}^m w_{(l)}$, where $w_{(l)}$ is the weight for $v_{(l)}$.
    \STATE Find the first index $m^*$ satisfying $S_{m^*} \ge T(\textbf{x})$.
    \STATE Set $\phi^{k+1}(\textbf{x}) = v_{(m^*)}$.
\ENDFOR

\STATE \textbf{Output:} Updated level-set function $\phi^{k+1}$.
\end{algorithmic}
\end{algorithm}

\subsection{Continuity Verification via the Chan-Vese Model}
\label{subsec:chan_vese_continuity}

We begin by applying our framework to the classical Chan-Vese (CV) model for bi-phasic image segmentation \cite{Chan_2001}. This experiment aims to demonstrate that our continuous formulation successfully resolves the discretization issues inherent in classical binary thresholding schemes.

\subsubsection{Model Formulation and Parameter Update}

Consider a grayscale image $I: \Omega \to \mathbbm{R}$. The CV model minimizes an energy functional balancing region-based fidelity and interface length. Within our general framework, the fidelity functions are:
\begin{equation}
    F_1(\textbf{x}, \textbf{c}) = |I(\textbf{x}) - c_1|^2, \quad F_2(\textbf{x}, \textbf{c}) = |I(\textbf{x}) - c_2|^2,
    \label{eq:cv_fidelity}
\end{equation}
where $\textbf{c} = (c_1, c_2)$ are the region constants. By fixing the level-set function $\phi^k$, the optimal parameters $c_1^{k+1}$ and $c_2^{k+1}$ are updated via the following recursive mean-value formulas:
\begin{equation}
    c_1^{k+1} = \frac{\int_{\Omega} \phi^k(\textbf{x}) I(\textbf{x}) d\textbf{x}}{\int_{\Omega} \phi^k(\textbf{x}) d\textbf{x}}, \quad c_2^{k+1} = \frac{\int_{\Omega} (1 - \phi^k(\textbf{x})) I(\textbf{x}) d\textbf{x}}{\int_{\Omega} (1 - \phi^k(\textbf{x})) d\textbf{x}}.
    \label{eq:c_recursion}
\end{equation}

\subsubsection{Comparison of Continuity: Small Time Step Behavior}

One of the primary motivations for adopting the continuous level-set formulation (the weighted quantile filter scheme) is to overcome the limitations of binary characteristic function formulation (the threshold dynamics scheme) on discrete grids. In the binary setting, where $u \in \{0, 1\}$, the evolution of the interface is governed by a simple thresholding rule \eqref{eq:threshold_update}. When the time-scale parameter $\tau$ is chosen to be very small, the interface often fails to move because the local forcing term $T(\textbf{x})$ is insufficient to ``flip'' the binary value of a pixel, leading to the pinning effect.

In contrast, our quantile filter approach evolves a continuous level-set function $\phi \in [0, 1]$ using the weighted quantile update \eqref{eq:median_update}. Because $\phi$ can take values between 0 and 1, the scheme effectively interpolates the interface position between grid points. This ``continuity'' allows the level-set function to undergo infinitesimal changes at each iteration, even when $\tau$ is extremely small.

Figure \ref{fig:pinning_effect} visually demonstrates this phenomenon. The experiment involves segmenting a noisy image containing multiple geometric shapes corrupted by Gaussian noise. For all simulations presented in this figure, the parameter is fixed at $\tilde{\lambda} = 0.6$. The figure compares the segmentation results of the traditional binary threshold dynamics (TD) scheme and our proposed quantile filter (QF) algorithm with varying time steps $\tau$. As $\tau$ decreases, the TD method exhibits increasing instability: it fails to handle topological changes (splitting) at intermediate steps and eventually freezes due to the pinning effect, failing to evolve from the initial contour (as seen at $\tau=10^{-4}$). In contrast, our MF method maintains a smooth and accurate interface evolution even for the smallest $\tau$ value, effectively verifying the ``continuity'' property of our level-set formulation.

\begin{figure}[htbp]
    \begin{subfigure}[b]{0.23\textwidth}
        \centering
        \includegraphics[width=\textwidth]{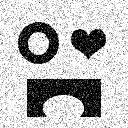} 
        \caption{Noisy Input}
    \end{subfigure}
    \hfill
    \begin{subfigure}[b]{0.23\textwidth}
        \centering
        \includegraphics[width=\textwidth]{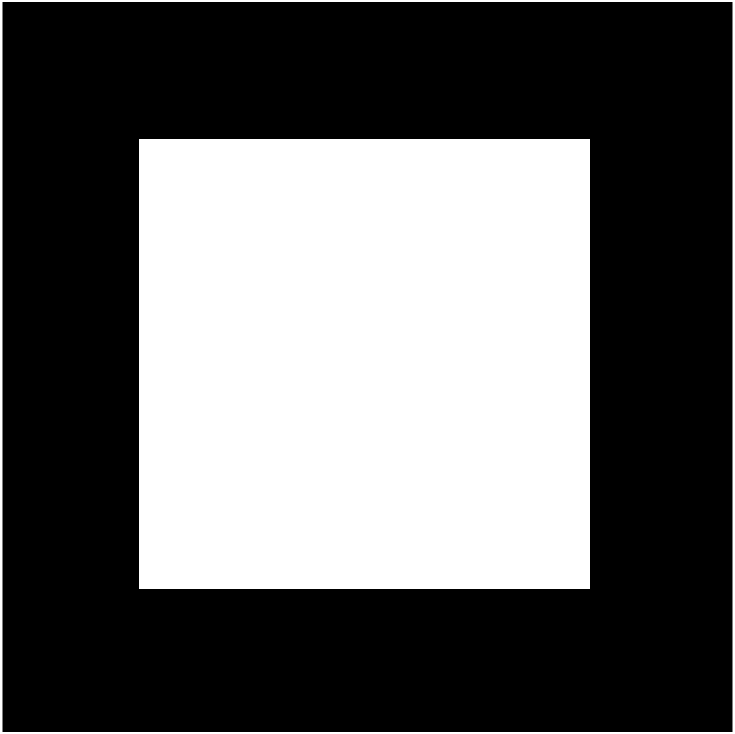} 
        \caption{Initial Contour}
    \end{subfigure}
    \hfill
    \begin{subfigure}[b]{0.23\textwidth}
        \centering
        \includegraphics[width=\textwidth]{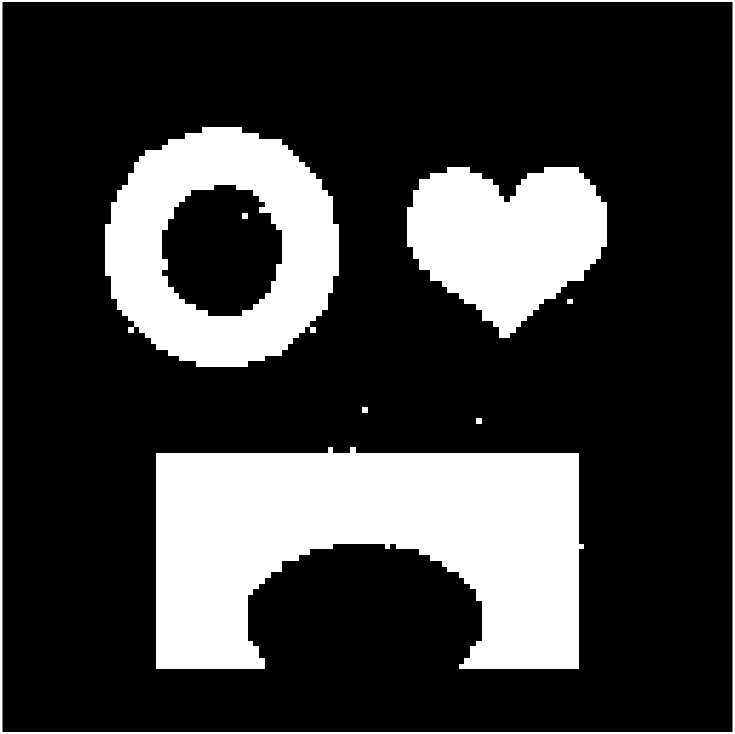}
        \caption{TD, $\tau=9 \times 10^{-4}$}
    \end{subfigure}
    \hfill
    \begin{subfigure}[b]{0.23\textwidth}
        \centering
        \includegraphics[width=\textwidth]{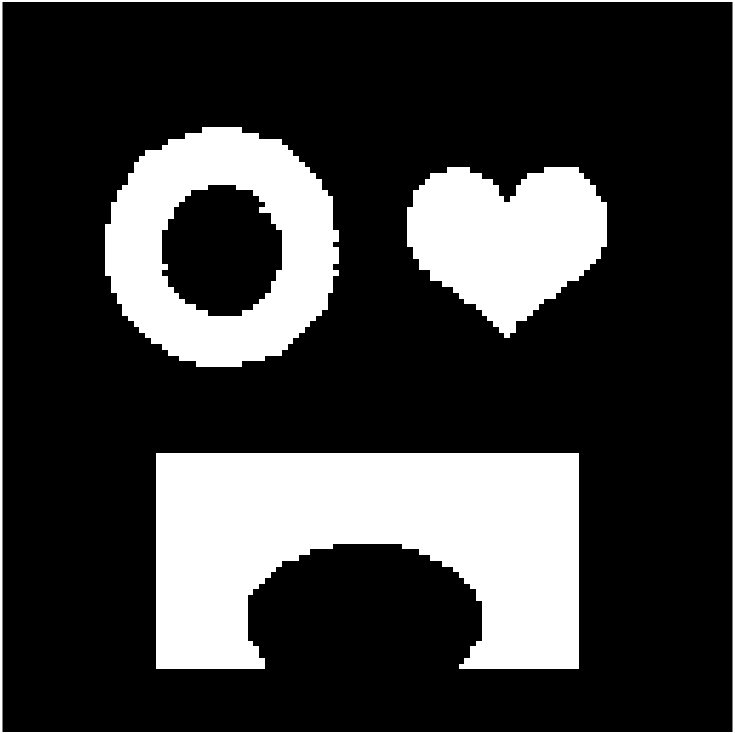}
        \caption{QF, $\tau=9 \times 10^{-4}$}
    \end{subfigure}
    
    \vspace{1em} 
    
    \begin{subfigure}[b]{0.23\textwidth}
        \centering
        \includegraphics[width=\textwidth]{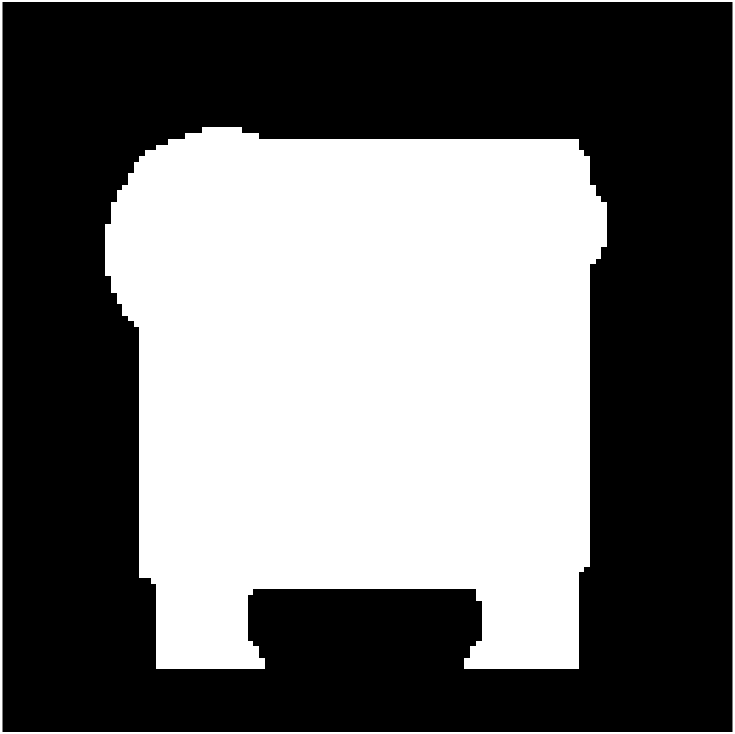}
        \caption{TD, $\tau=7 \times 10^{-4}$}
    \end{subfigure}
    \hfill
    \begin{subfigure}[b]{0.23\textwidth}
        \centering
        \includegraphics[width=\textwidth]{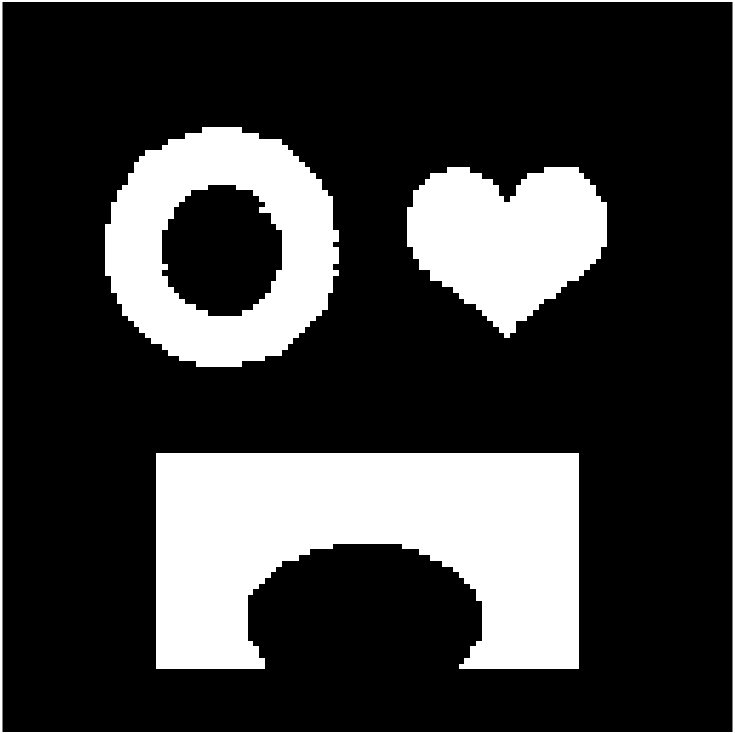}
        \caption{QF, $\tau=7 \times 10^{-4}$}
    \end{subfigure}
    \hfill
    \begin{subfigure}[b]{0.23\textwidth}
        \centering
        \includegraphics[width=\textwidth]{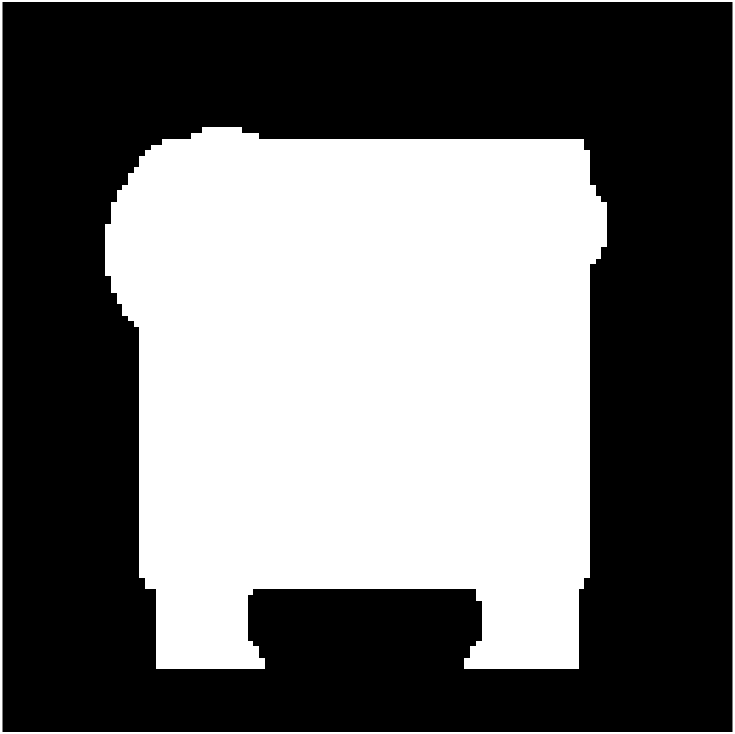}
        \caption{TD, $\tau=5 \times 10^{-4}$}
    \end{subfigure}
    \hfill
    \begin{subfigure}[b]{0.23\textwidth}
        \centering
        \includegraphics[width=\textwidth]{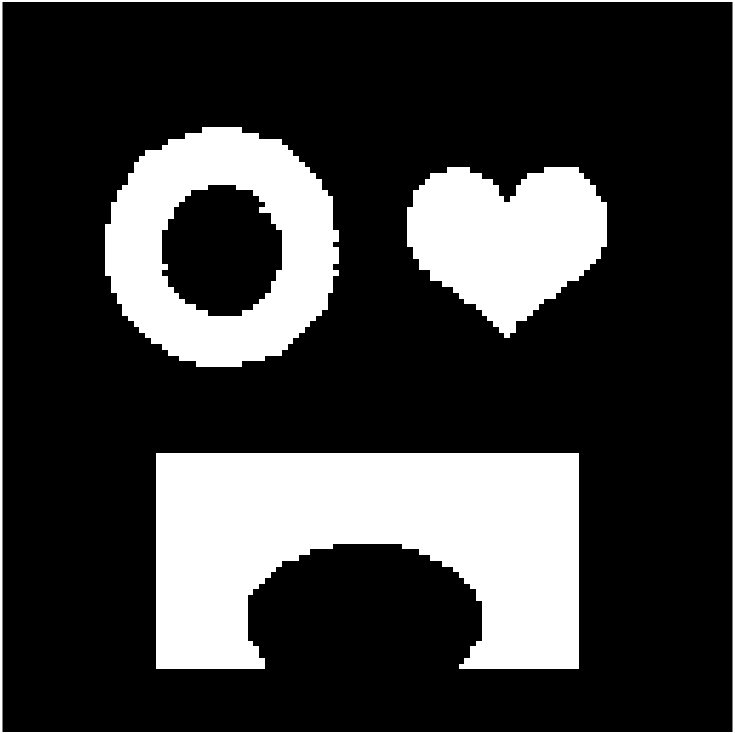}
        \caption{QF, $\tau=5 \times 10^{-4}$}
    \end{subfigure}
    
    \vspace{1em}

    \begin{subfigure}[b]{0.23\textwidth}
        \centering
        \includegraphics[width=\textwidth]{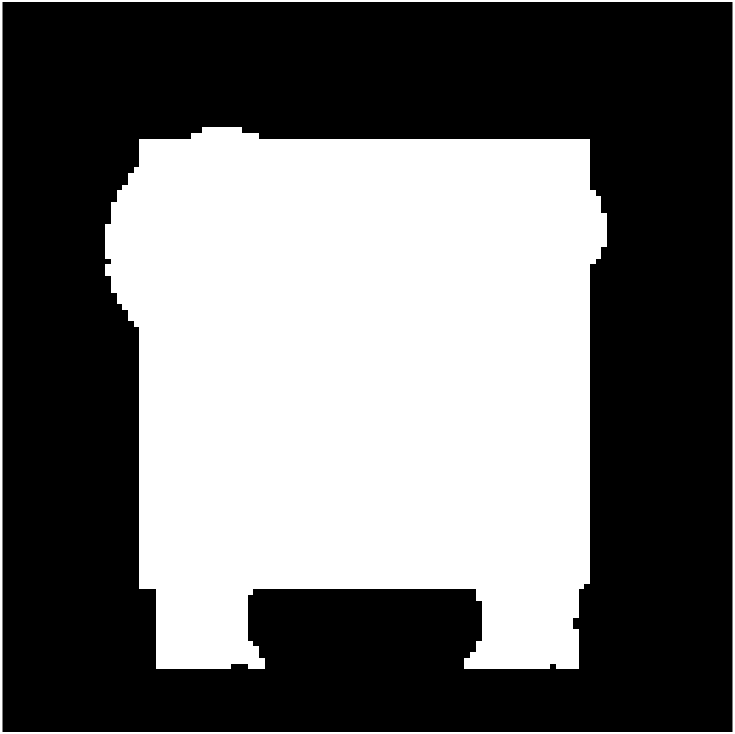}
        \caption{TD, $\tau=3 \times 10^{-4}$}
    \end{subfigure}
    \hfill
    \begin{subfigure}[b]{0.23\textwidth}
        \centering
        \includegraphics[width=\textwidth]{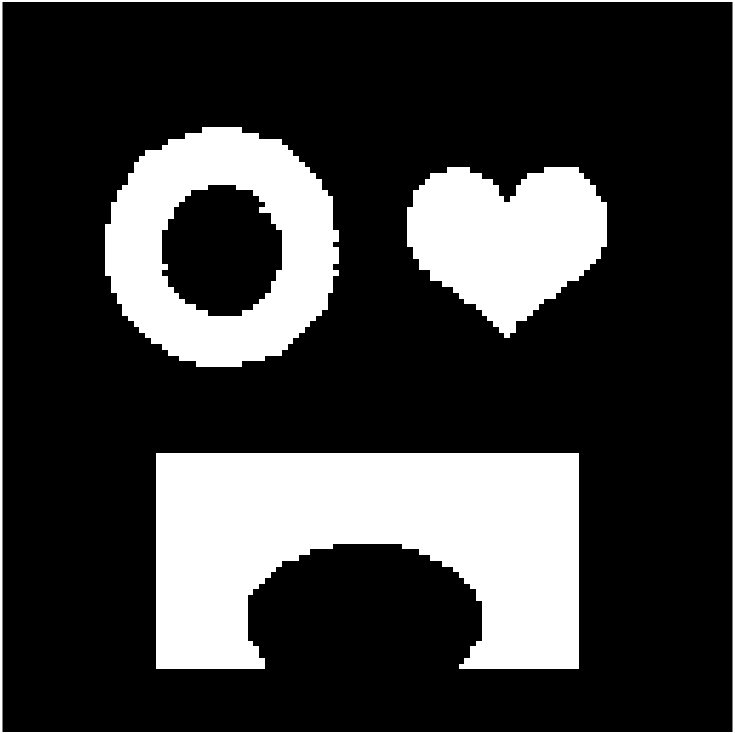}
        \caption{QF, $\tau=3 \times 10^{-4}$}
    \end{subfigure}
    \hfill
    \begin{subfigure}[b]{0.23\textwidth}
        \centering
        \includegraphics[width=\textwidth]{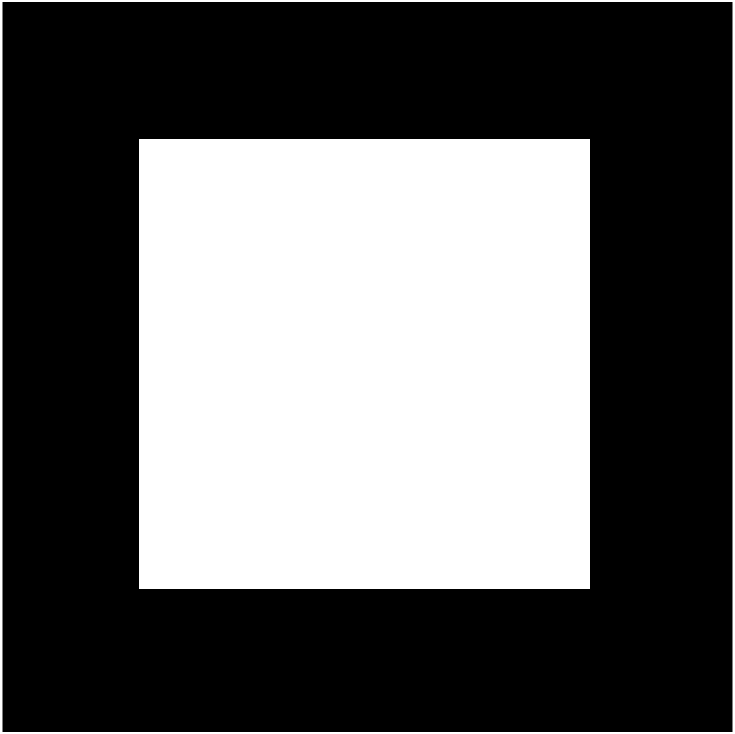}
        \caption{TD, $\tau=10^{-4}$}
    \end{subfigure}
    \hfill
    \begin{subfigure}[b]{0.23\textwidth}
        \centering
        \includegraphics[width=\textwidth]{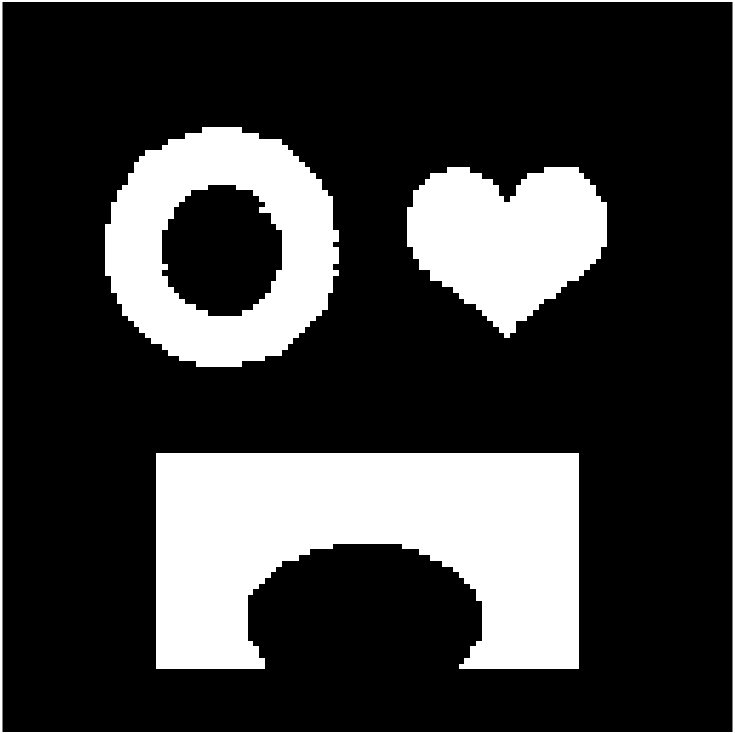}
        \caption{QF, $\tau=10^{-4}$}
    \end{subfigure}
    
    \caption{Visual comparison of segmentation results on a noisy image with different time steps $\tau$. The parameter is set to $\tilde{\lambda} = 0.6$. 
    (a) The noisy input image. 
    (b) The initial square contour. 
    (c)--(l) Comparison between the traditional threshold dynamics (TD) method and the proposed quantile filter (QF) scheme with decreasing time steps $\tau \in \{9, 7, 5, 3, 1\} \times 10^{-4}$. 
    It can be observed that the QF scheme (d, f, h, j, l) consistently achieves accurate segmentation regardless of the time step size. 
    In contrast, the TD method demonstrates significant instability: at intermediate time steps (e, g, i), it fails to handle topological changes (splitting); and at a small time step (k), it suffers from the severe pinning effect, where the contour fails to evolve from the initialization (compare k with b).}
    \label{fig:pinning_effect}
\end{figure}

\begin{remark}
\label{rem:param_guide}
The selection of the key parameters $\tau$ and $\tilde{\lambda}$ can be guided by the empirical behaviors observed in the numerical experiments, particularly in Figure \ref{fig:pinning_effect}:

\paragraph{Selection of $\tau$}
The parameter $\tau$ determines the time step size and the effective interaction radius ($R = \sqrt{2\tau}$). Unlike traditional threshold dynamics, the proposed continuous framework accommodates an arbitrary $\tau$ without suffering from the pinning effect. As empirically validated in Figure \ref{fig:pinning_effect}, a relatively large $\tau$ (e.g., $\tau = 9 \times 10^{-4}$) accelerates the morphological evolution, yielding high computational efficiency. Conversely, a smaller $\tau$ (e.g., $\tau = 10^{-4}$) is preferable for achieving high geometric accuracy and preserving sharp corners, as it effectively minimizes the artificial rounding effect inherent to the Gaussian kernel.

\paragraph{Selection of $\tilde{\lambda}$}
The parameter $\tilde{\lambda}$ controls the trade-off between data fidelity and interface smoothness, where a larger $\tilde{\lambda}$ imposes a stronger penalty on the perimeter. As demonstrated in the noisy image segmentation in Figure \ref{fig:pinning_effect}, a moderate value (e.g., $\tilde{\lambda} = 0.6$) is empirically effective for smoothing out substantial noise while preserving essential geometric features. In broader applications, if the resulting interface appears too fragmented or retains excessive noise, $\tilde{\lambda}$ can be incrementally increased to enforce stronger geometric regularization.
\end{remark}

\subsection{Interface Sharpness Verification via the Chan-Vese Model}
\label{subsec:chan_vese_sharpness}

In this section, we conduct numerical experiments to validate the theoretical assertions regarding the interface sharpening property. The primary objective is to demonstrate that the proposed quantile filter scheme effectively drives the interface from a continuous initialization to a sharp profile without any artificial post-processing or hard thresholding steps at the end of the algorithm, and converges to the geometric steady state predicted by the analysis in Theorem \ref{thm:binary_enforcement}.

The numerical validation is performed with a fixed time step $\tau = 10^{-3}$ and the parameter $\tilde{\lambda} = 0.6$. The computational domain $\Omega$ is discretized on a uniform grid. The initialization is chosen as a regularized signed distance function with a conical profile. Specifically, it is normalized to have a maximum value of $\phi=1$ at the domain center and linearly decays to $\phi=0$ at the domain corners. We observe its morphological evolution and the recovery of sharpness as it approaches the stationary state.

Figure \ref{fig:sharpness_verification} presents a comprehensive visualization of the interface evolution, organized into two perspectives to verify the theoretical predictions. The top row (3D views) displays the morphological evolution of the interface level set function $\phi$, capturing the spatial transition from the conical initialization to the sharp, cylindrical steady state. Concurrently, the bottom row (histograms) illustrates the statistical evolution of the function values. As time progresses, the distribution shifts from a broad spread across the interval $[0,1]$ to a binary concentration at the endpoints $\{0, 1\}$, providing quantitative evidence of the sharpening effect and the suppression of numerical diffusion.

\begin{figure}[htbp]
    \centering
    \begin{subfigure}[b]{0.24\textwidth}
        \centering
        \includegraphics[width=\textwidth]{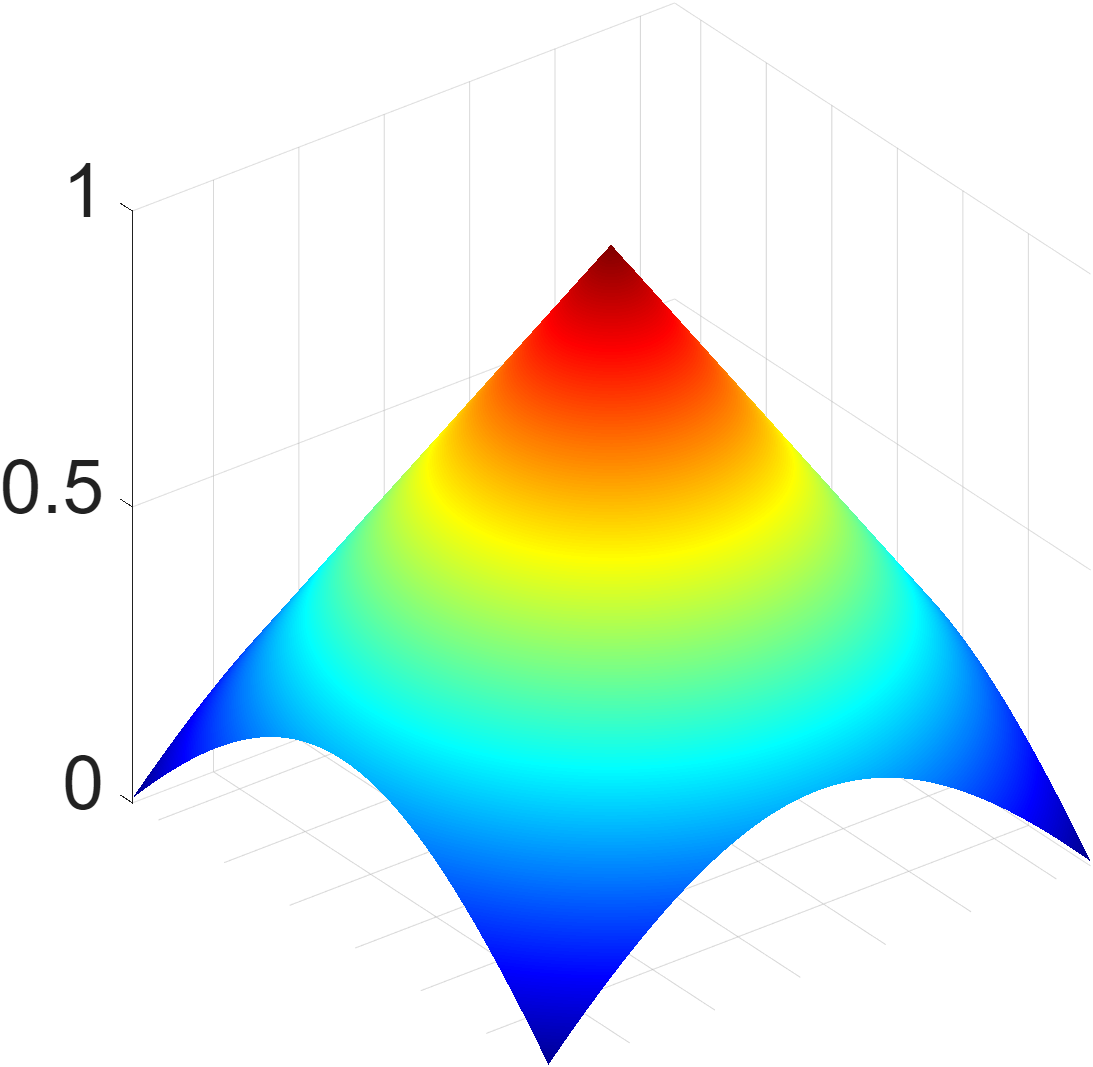} \caption{Initial State}
    \end{subfigure}
    \hfill
    \begin{subfigure}[b]{0.24\textwidth}
        \centering
        \includegraphics[width=\textwidth]{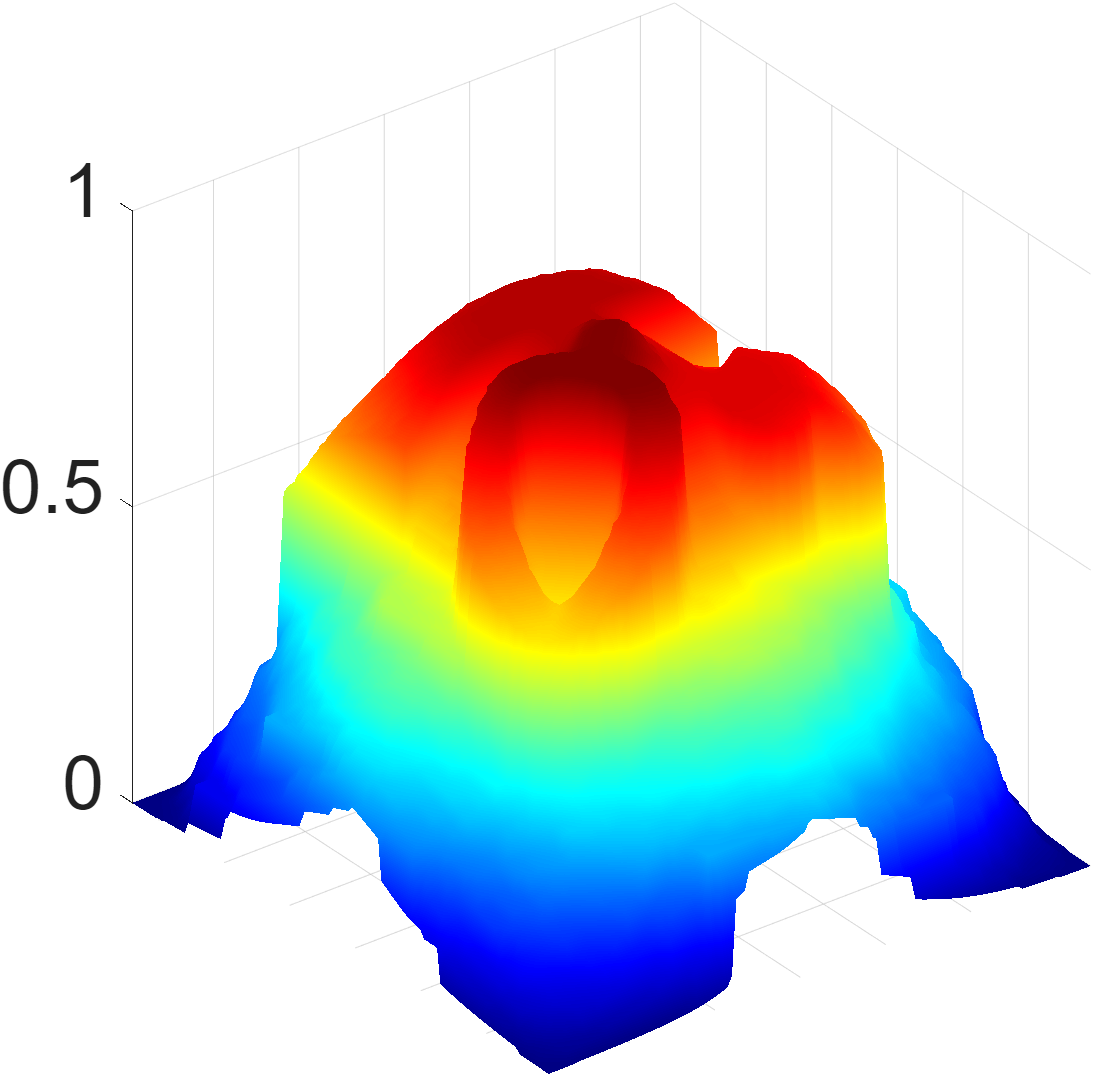} 
        \caption{Early Evolution}
    \end{subfigure}
    \hfill
    \begin{subfigure}[b]{0.24\textwidth}
        \centering
        \includegraphics[width=\textwidth]{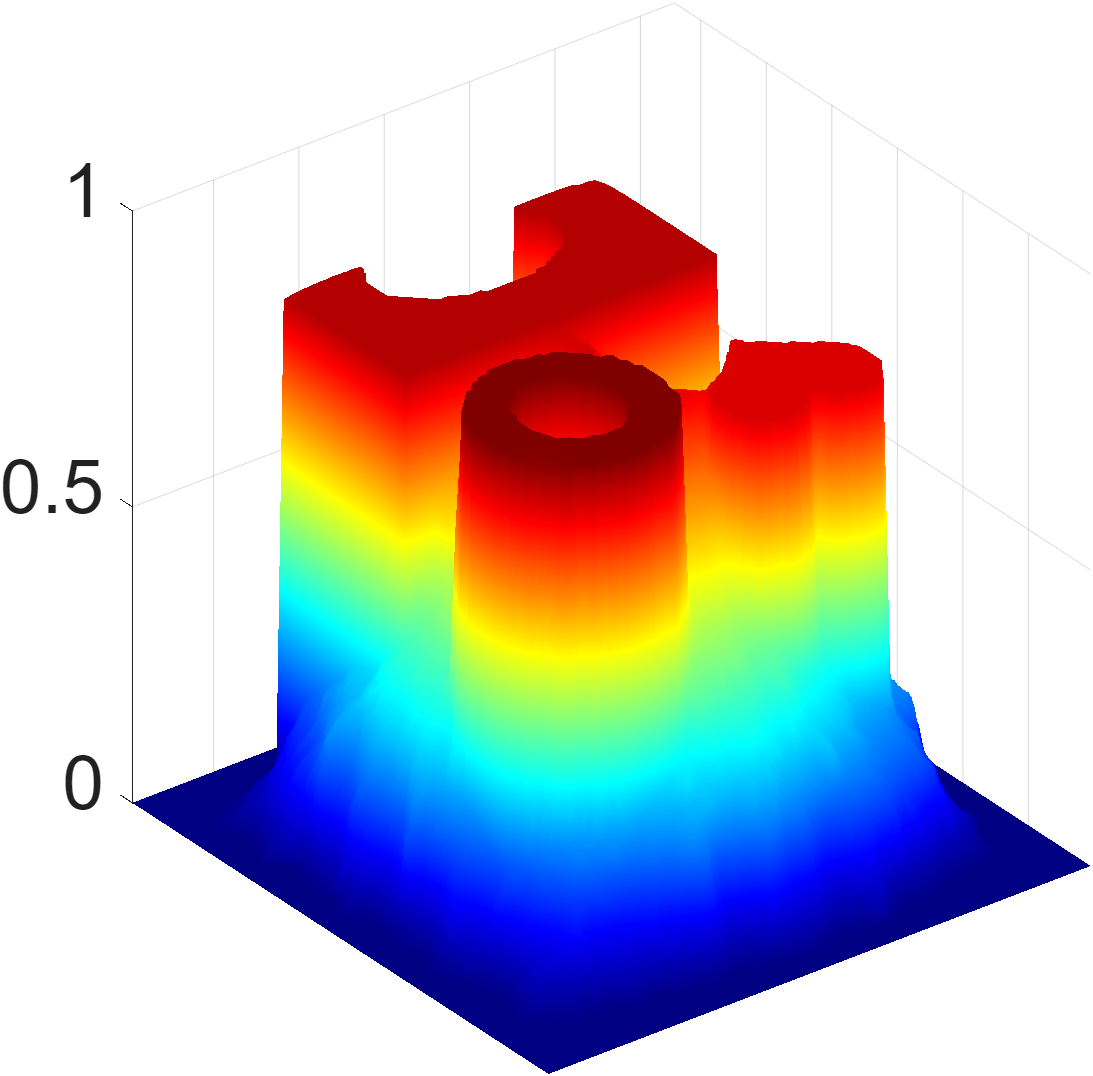} 
        \caption{Late Evolution}
    \end{subfigure}
    \hfill
    \begin{subfigure}[b]{0.24\textwidth}
        \centering
        \includegraphics[width=\textwidth]{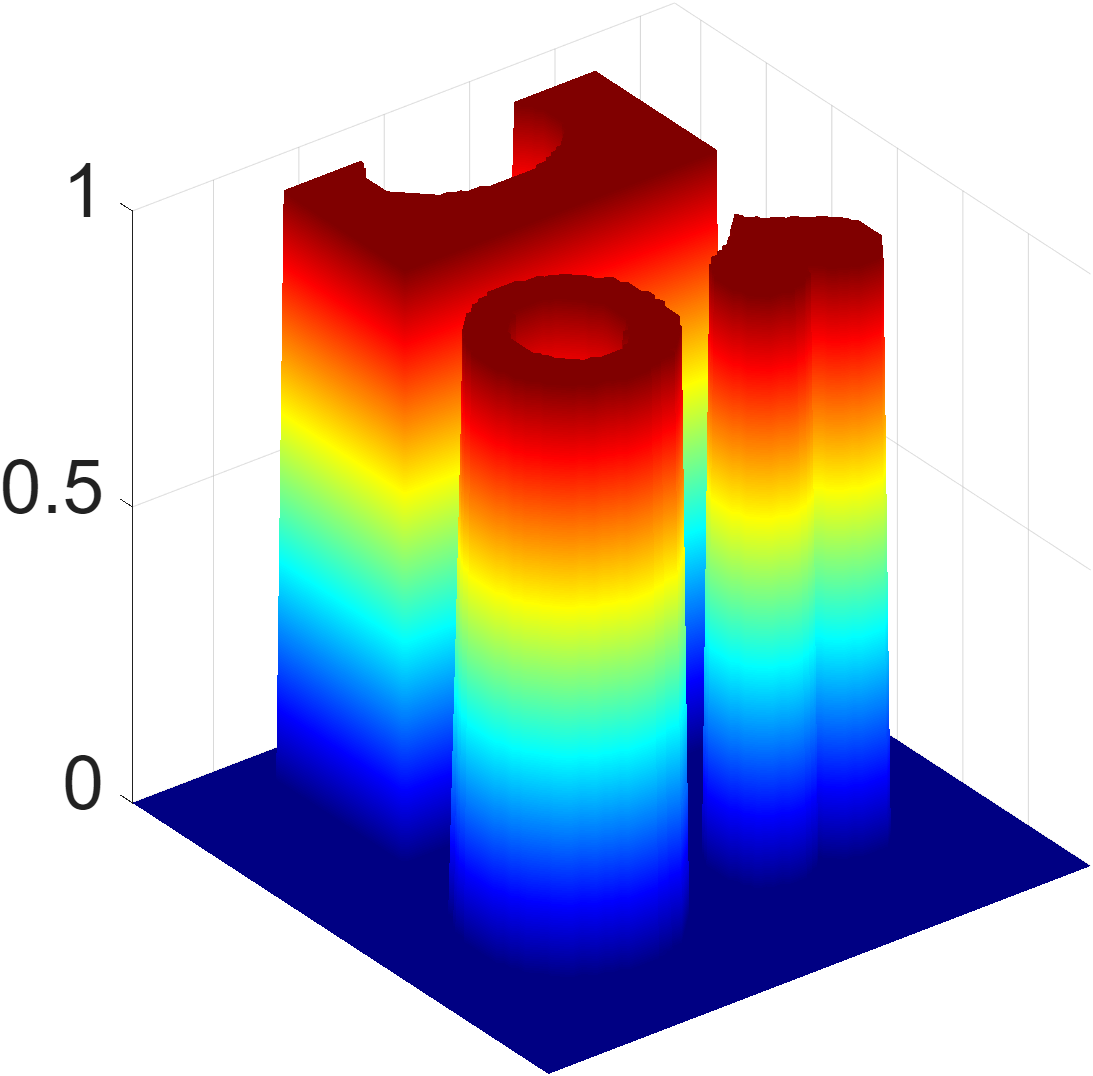} 
        \caption{Steady State}
    \end{subfigure}
    
    \vspace{1em} 
    
    \begin{subfigure}[t]{0.24\textwidth}
        \centering
        \includegraphics[width=\textwidth]{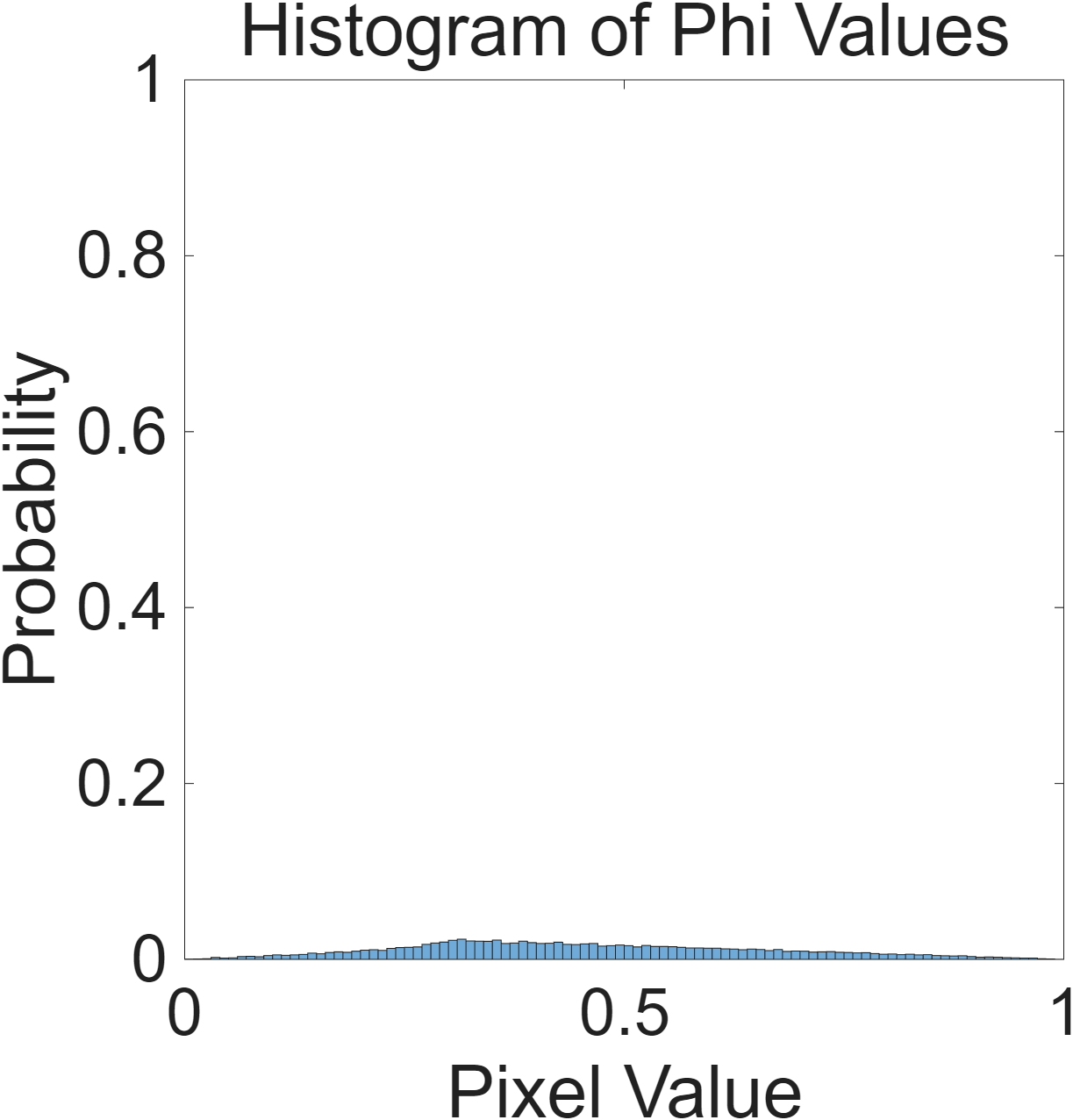} 
        \caption{Initial State}
    \end{subfigure}
    \hfill
    \begin{subfigure}[t]{0.24\textwidth}
        \centering
        \includegraphics[width=\textwidth]{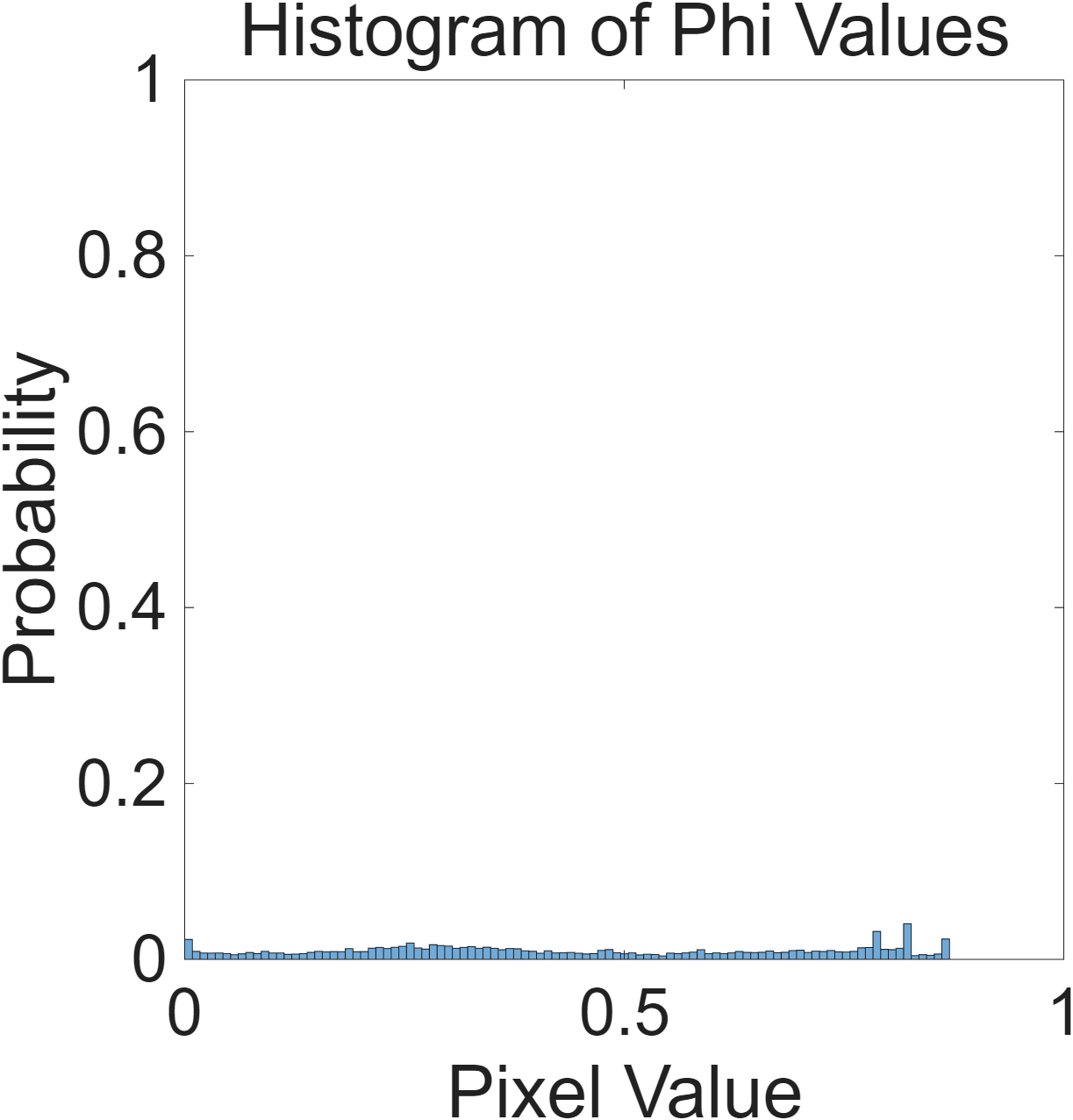}
        \caption{Early Evolution}
    \end{subfigure}
    \hfill
    \begin{subfigure}[t]{0.24\textwidth}
        \centering
        \includegraphics[width=\textwidth]{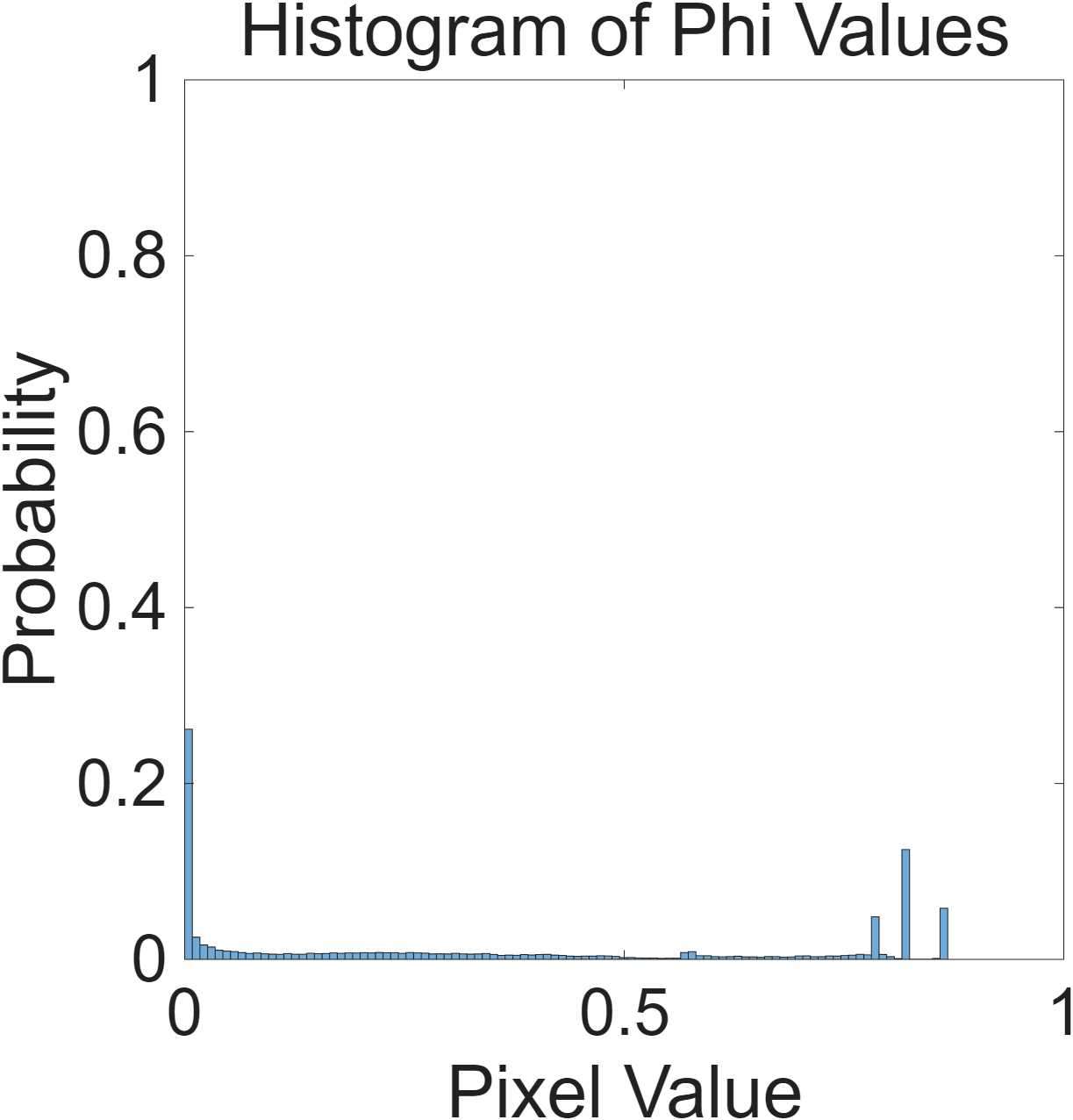}
        \caption{Late Evolution}
    \end{subfigure}
    \hfill
    \begin{subfigure}[t]{0.24\textwidth}
        \centering
        \includegraphics[width=\textwidth]{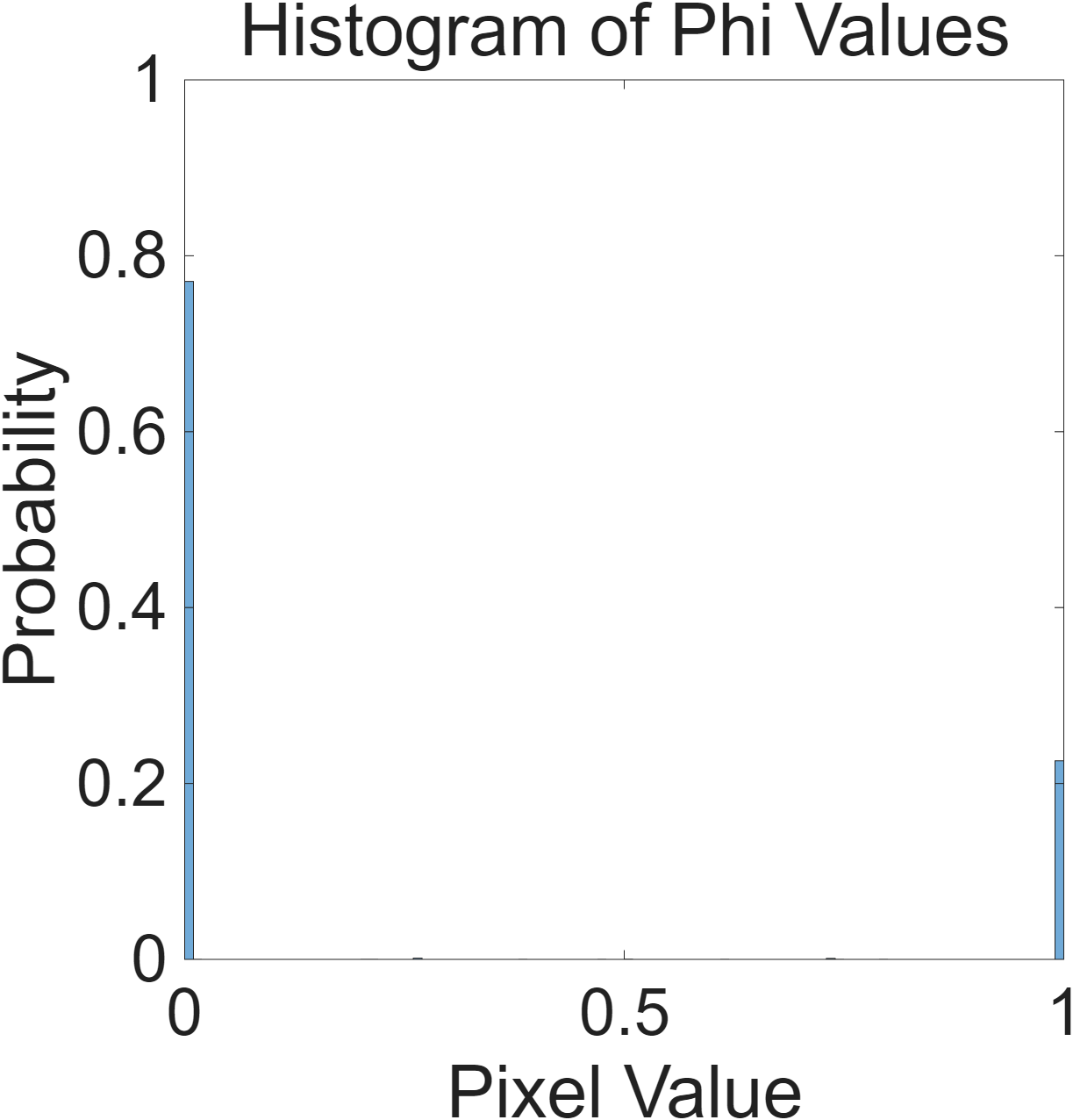}
        \caption{Steady State}
    \end{subfigure}
    
    \caption{Visual verification of sharpness. Top row: 3D evolution of the interface level set function $\phi$ showing the sharpening process. Bottom row: Corresponding histograms of pixel values showing the transition from a continuous distribution to a binary state. The columns represent the progression from initialization (left) to the final geometric steady state (right).}
    \label{fig:sharpness_verification}
\end{figure}

\subsubsection{High-Order Implementation via Quadratic Interpolation}
    To improve geometric accuracy, we propose the following quadratic interpolation algorithm, which effectively captures the local geometric features of the image.

\begin{algorithm}[H]
    \caption{Continuous Quantile Filter Scheme via Quadratic Interpolation}
    \label{alg:quadratic_interpolation}
    \begin{algorithmic}[1]
        \STATE \textbf{Input:} Current level-set function $\phi^k$, time step $\tau$.
        
        \FOR{each grid point $\mathbf{x}$ in the domain}
            \STATE Sample neighbors $\{\phi^k(\mathbf{x}+\mathbf{y}_j)\}_{j=1}^M$ on a circle with center $\textbf{x}$ and radius $\tau$.
            \STATE Construct local quadratic polynomial $\mathcal{P}(\theta)$ approximate values on the circle.
            \STATE Find $\phi^*$ such that $|\{\theta \in [0, 2\pi) : \mathcal{P}(\theta) < \phi^*\}| = 2\pi T(\mathbf{x})$.
            \STATE Set $\phi^{k+1}(\mathbf{x}) = \phi^*$.
        \ENDFOR
        
        \STATE \textbf{Output:} Updated level-set function $\phi^{k+1}$.
    \end{algorithmic}
\end{algorithm}

\begin{remark}
    In practice, to achieve sub-grid accuracy, we discretize the neighborhood boundary into $M=8$ equidistant sampling points. These nodes partition the domain into 4 non-overlapping segments. On each segment $i \in \{1, \dots, 4\}$, the local level-set function is approximated by a quadratic polynomial $\mathcal{P}_{i}(\theta) = c_{2,i} \theta^2 + c_{1,i} \theta + c_{0,i}$, uniquely determined by the nodal values at the two endpoints and one midpoint.
\end{remark}

\begin{figure}[htbp]
    \centering
    \begin{subfigure}[b]{0.24\textwidth}
        \centering
        \includegraphics[width=\textwidth]{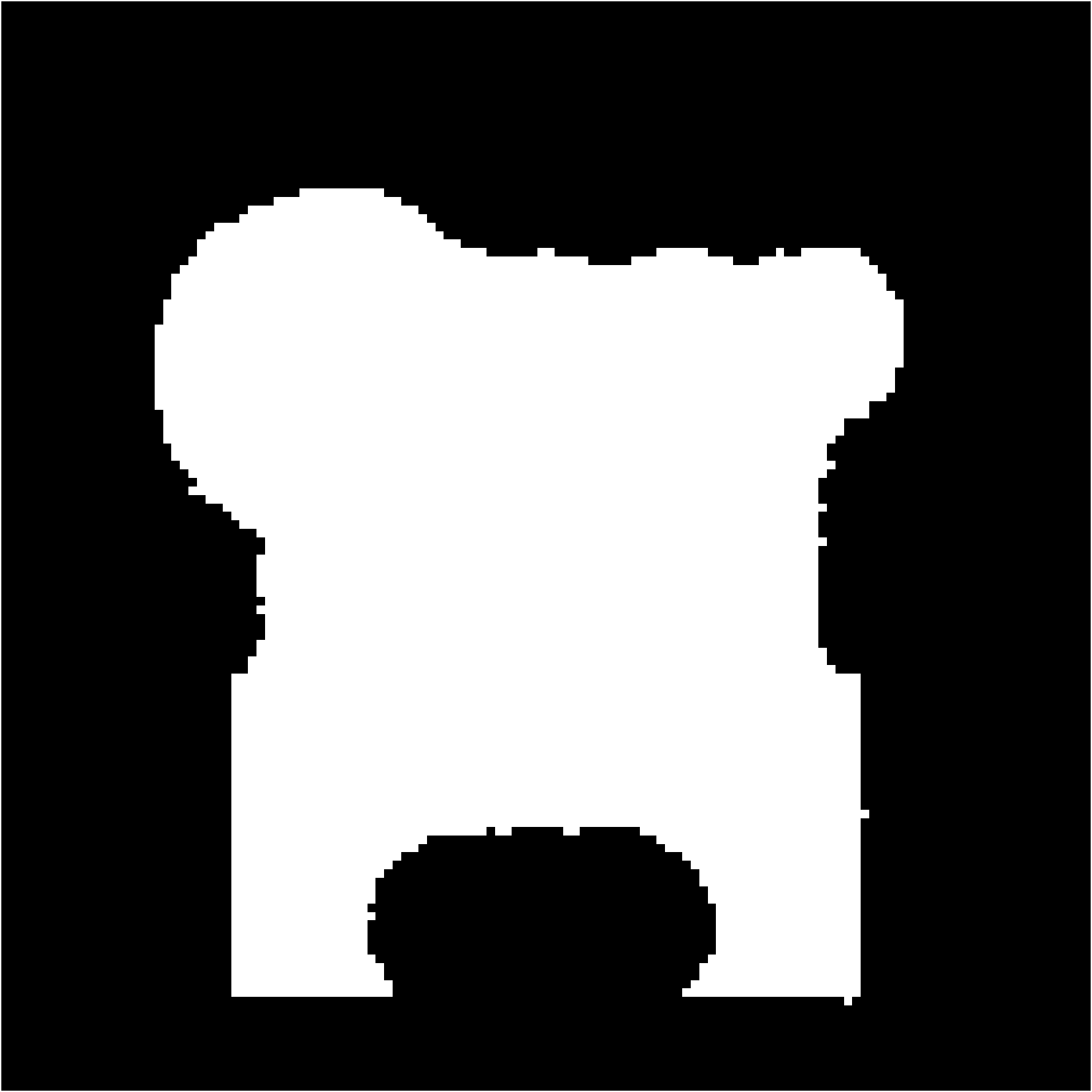} 
        \caption{Inter ($k=20$)}
        \label{fig:quad_step1}
    \end{subfigure}
    \hfill
    \begin{subfigure}[b]{0.24\textwidth}
        \centering
        \includegraphics[width=\textwidth]{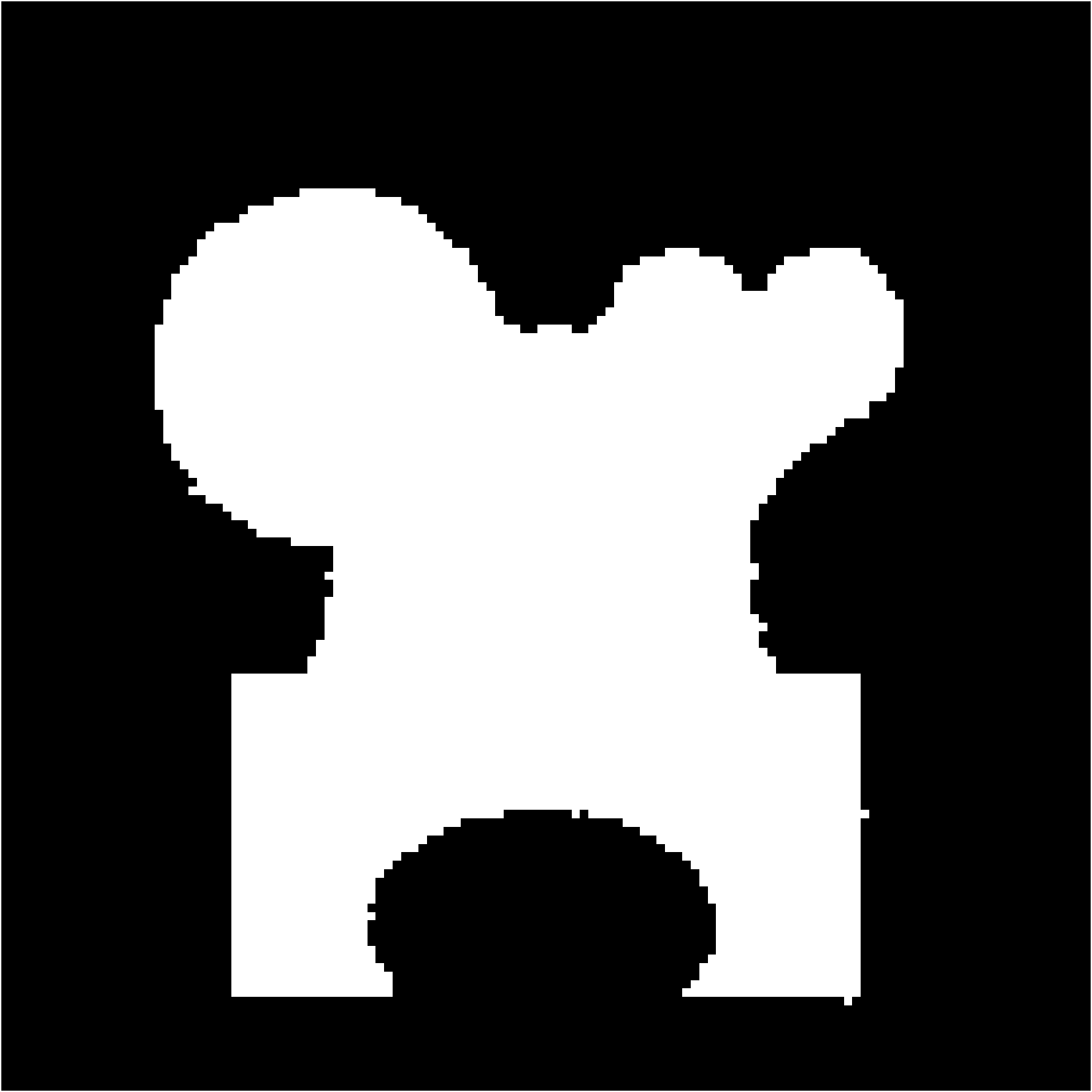} 
        \caption{Inter ($k=40$)} 
        \label{fig:quad_step2}
    \end{subfigure}
    \hfill
    \begin{subfigure}[b]{0.24\textwidth}
        \centering
        \includegraphics[width=\textwidth]{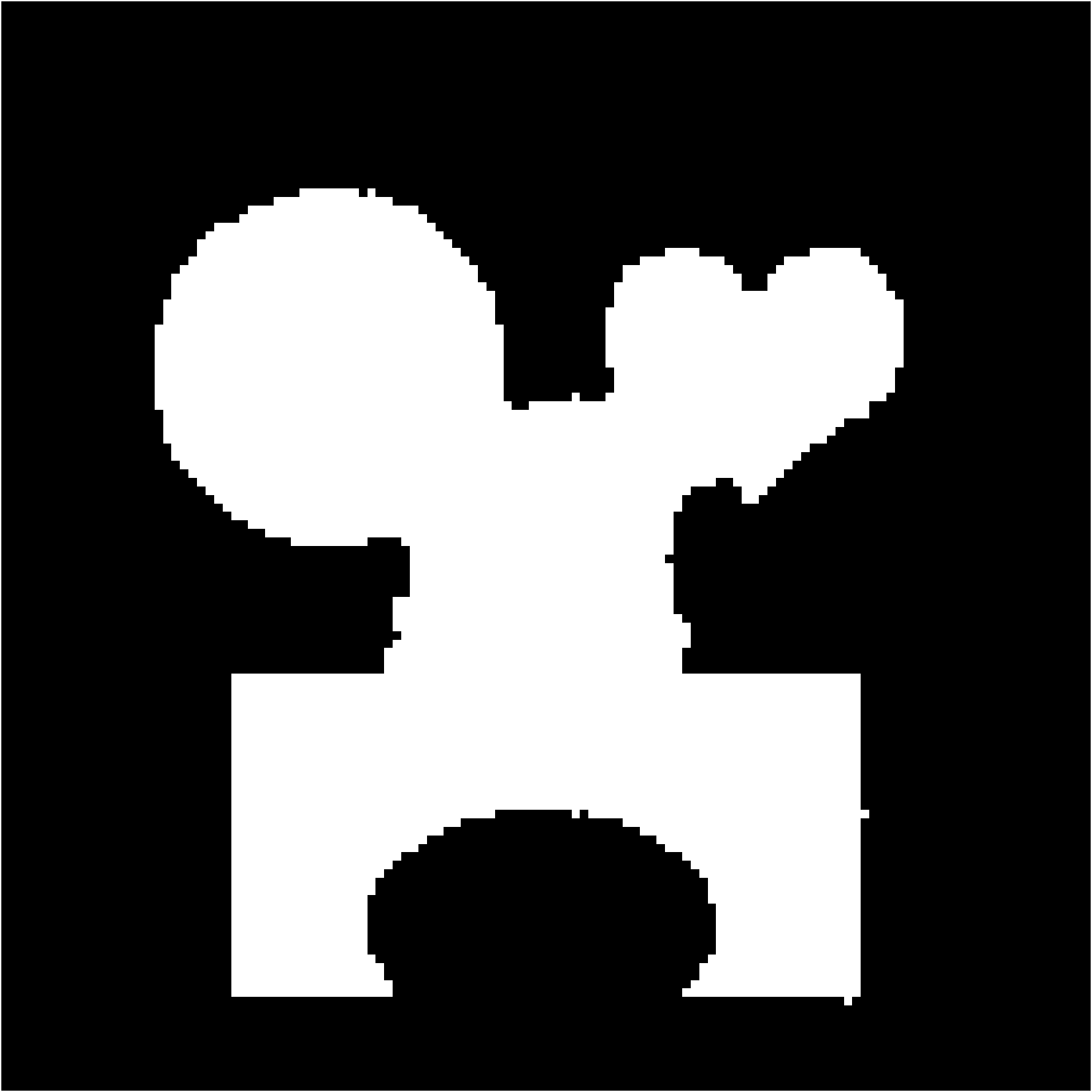} 
        \caption{Inter ($k=60$)} 
        \label{fig:quad_step3}
    \end{subfigure}
    \hfill
    \begin{subfigure}[b]{0.24\textwidth}
        \centering
        \includegraphics[width=\textwidth]{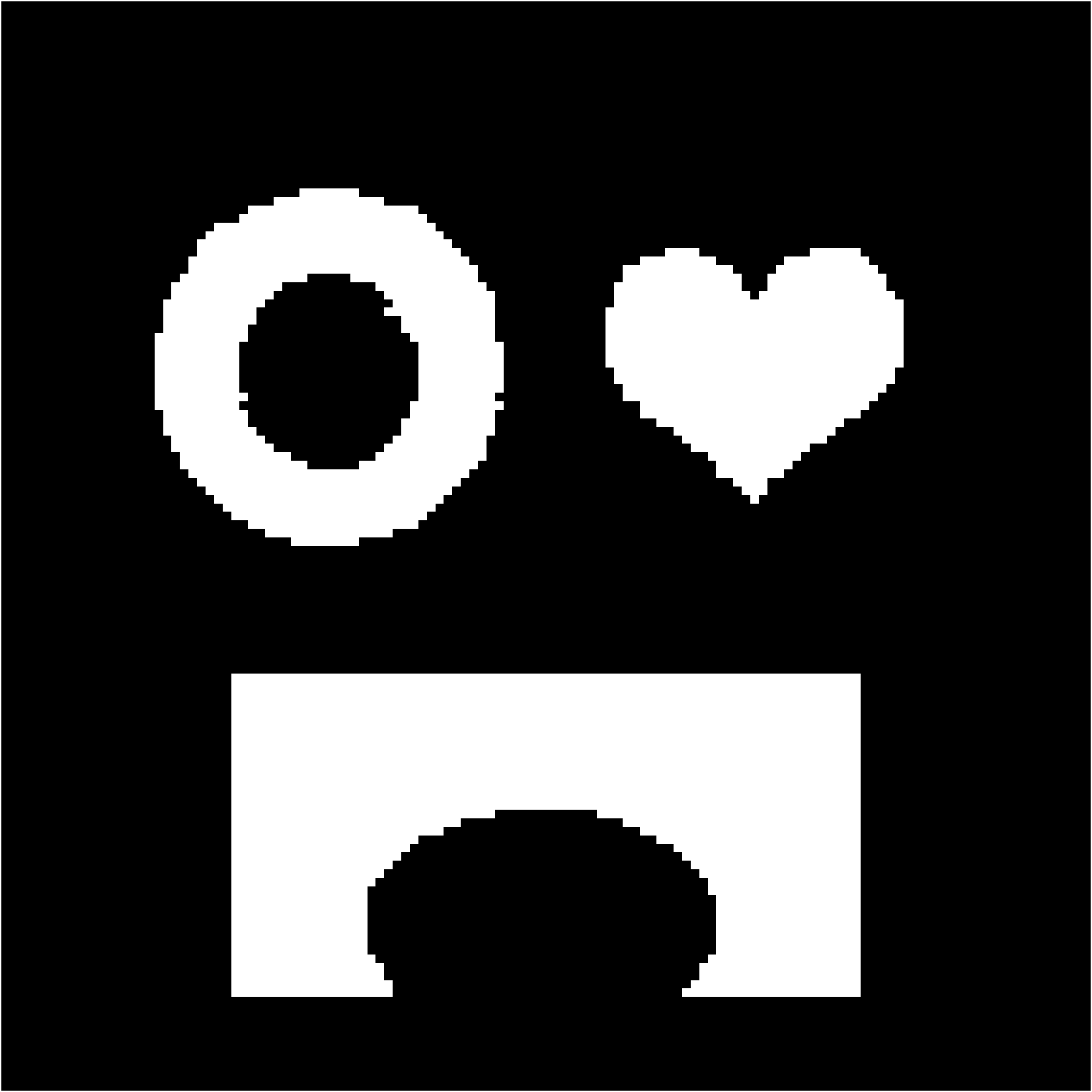} 
        \caption{Final ($k=90$)}
        \label{fig:quad_step4}
    \end{subfigure}
    
    \caption{The simulation is performed with a time step $\tau=5 \times 10^{-4}$, and the effective parameter $\tilde{\lambda}=0.6$. The sequence illustrates the robust evolution of the contour, effectively capturing geometric details and handling topological changes, finally converging to a sharp binary solution at iteration $k=90$.}
    \label{fig:quad_evolution}
\end{figure}

While the Chan-Vese model performs well on images with distinct, homogeneous regions, it relies on the assumption that the intensity within each phase is approximately piecewise constant. This assumption fails in the presence of intensity inhomogeneity (e.g., uneven illumination), where a global fitting term is insufficient. This limitation motivates the need for more sophisticated energy functionals. In the next section, we demonstrate that our weighted-quantile-filter-based framework is flexible enough to handle such complex scenarios by employing the local intensity fitting (LIF) model.

\subsection{Stability Verification via Local Intensity Fitting (LIF)}
\label{subsec:stability_LIF}
To rigorously test the stability of our algorithm (Theorem \ref{thm:unconditional_stability}) under more complex energy functionals, we employ the local intensity fitting (LIF) model \cite{Li_2007, Chunming_Li_2008}. Unlike the Chan-Vese model which assumes global constant intensities, the LIF model captures the local intensity information of the image. Specifically, it is designed to deal with intensity inhomogeneity by minimizing the smoothed error between the original image and the local fitting image. This capability is essential for segmenting images with severe intensity inhomogeneity and noise.

\subsubsection{Model Formulation and Parameter Update}

Consider a grayscale image $I: \Omega \to \mathbbm{R}$. The LIF model minimizes an energy functional that handles intensity inhomogeneities by using local region statistics. The fidelity functional is defined as:
\begin{equation}
    F_i(\textbf{y}, \textbf{c}) = \int_{\Omega} G_\sigma(\textbf{x}-\textbf{y}) |C_i(\textbf{x}) - I(\textbf{y})|^2 d\textbf{x}, \quad i = 1, 2, 
    \label{eq:lif_energy}
\end{equation}
where $G_\sigma$ is a Gaussian kernel with variance $2\sigma$. Unlike the Chan-Vese model where $\mathbf{c}$ consists of global constants, in the LIF model, the generalized parameter vector $\mathbf{c}$ comprises two spatially varying functions, i.e., $\mathbf{c} = (C_1(\mathbf{x}), C_2(\mathbf{x}))$, where $C_i(\mathbf{x})$ represents the local intensity means associated with each phase.

By fixing the intermediate level-set function $\phi^k$, the optimal local means $C_1$ and $C_2$ satisfy the Euler-Lagrange equations and are updated via the following convolution formulas:
\begin{equation}
    C_1^{k+1}(\textbf{x}) = \frac{(G_\sigma * (\phi^k I))(\textbf{x})}{(G_\sigma * \phi^k)(\textbf{x})}, \quad C_2^{k+1}(\textbf{x}) = \frac{(G_\sigma * ((1 - \phi^k) I))(\textbf{x})}{(G_\sigma * (1 - \phi^k))(\textbf{x})}.
    \label{eq:lif_update}
\end{equation}
Here, $*$ denotes the convolution operation.

\subsubsection{Validation of Stability: Energy Decay}

The validation of the unconditional stability property focuses on the monotonicity of the energy evolution within our weighted-quantile-filter-based framework. Figure \ref{fig:lif_experiment} presents the segmentation process in an image characterized by severe intensity inhomogeneity, serving as a stress test for the LIF model.

\begin{figure}[htbp]
    \centering
    \begin{subfigure}[b]{0.24\textwidth}
        \centering
        \includegraphics[width=\textwidth]{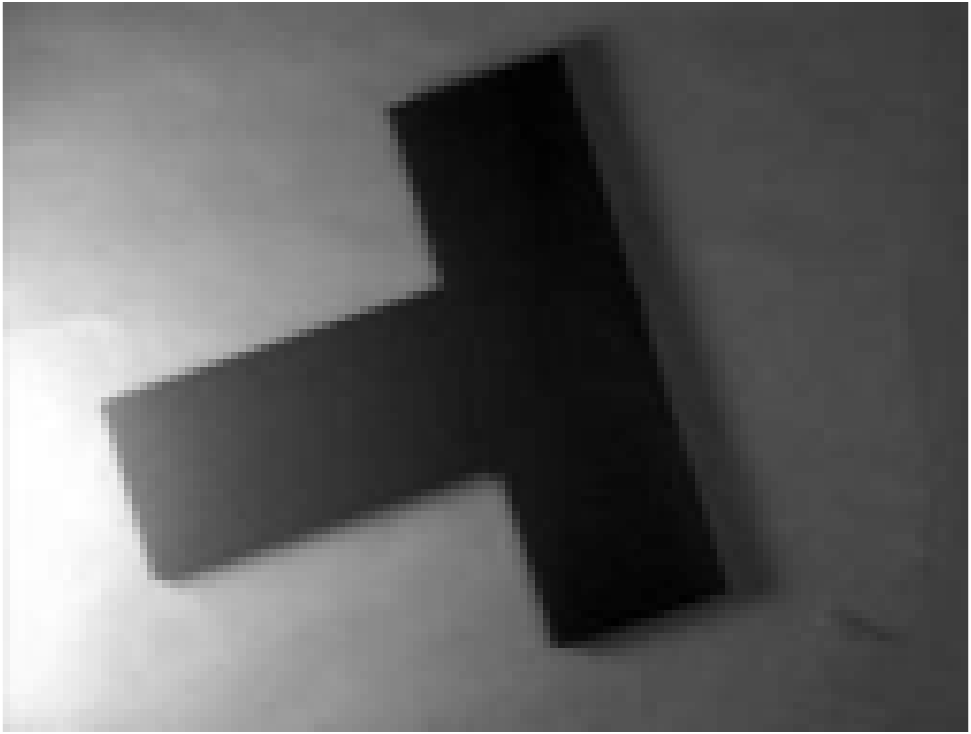}
        \vspace{-0.2em}
        \caption{Image}
        \label{fig:lif_origin}
    \end{subfigure}
    \hfill
    \begin{subfigure}[b]{0.24\textwidth}
        \centering
        \includegraphics[width=\textwidth]{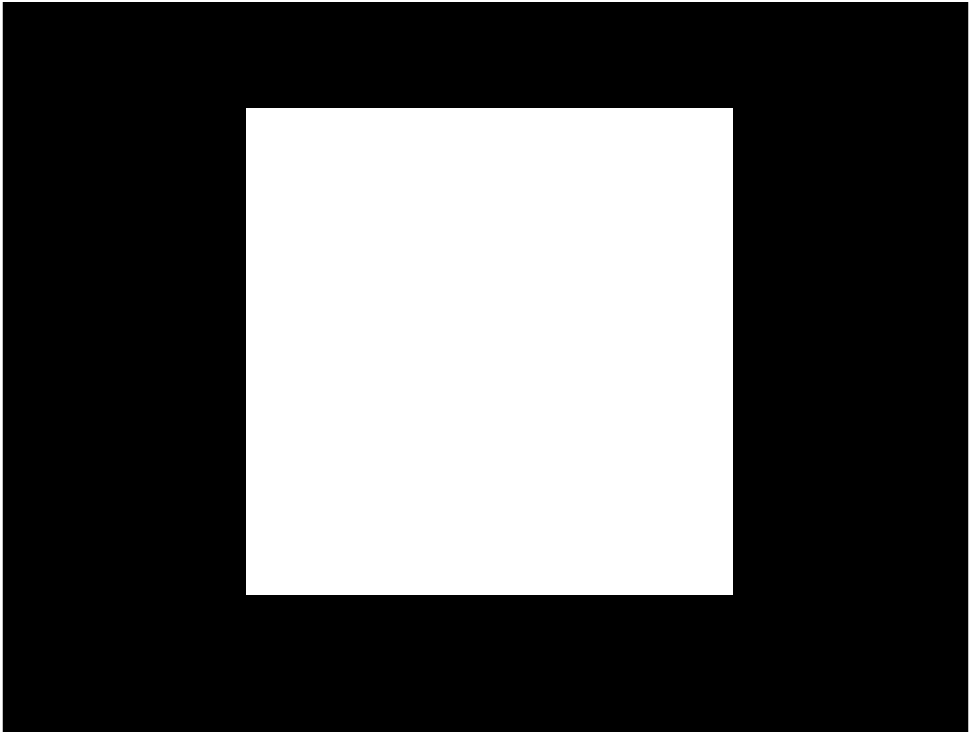}
        \vspace{-0.2em}
        \caption{Initial State}
        \label{fig:lif_initial}
    \end{subfigure}
    \hfill
    \begin{subfigure}[b]{0.24\textwidth}
        \centering
        \includegraphics[width=\textwidth]{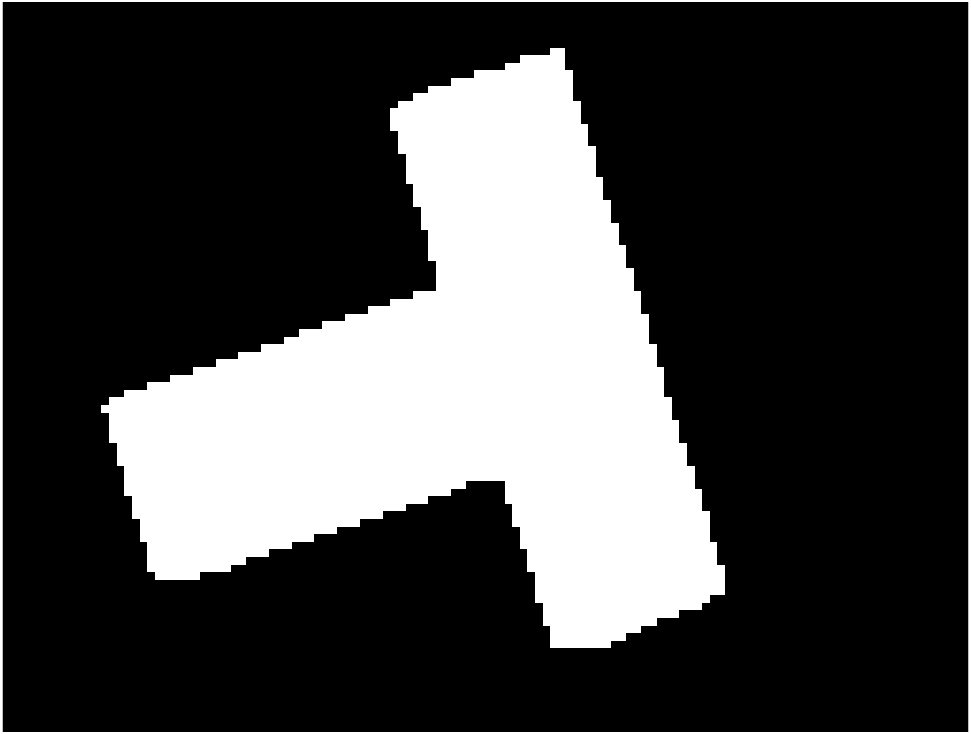}
        \vspace{-0.2em}
        \caption{Steady State}
        \label{fig:lif_result}
    \end{subfigure}
    \hfill
    \begin{subfigure}[b]{0.24\textwidth}
        \centering
        \includegraphics[width=\textwidth]{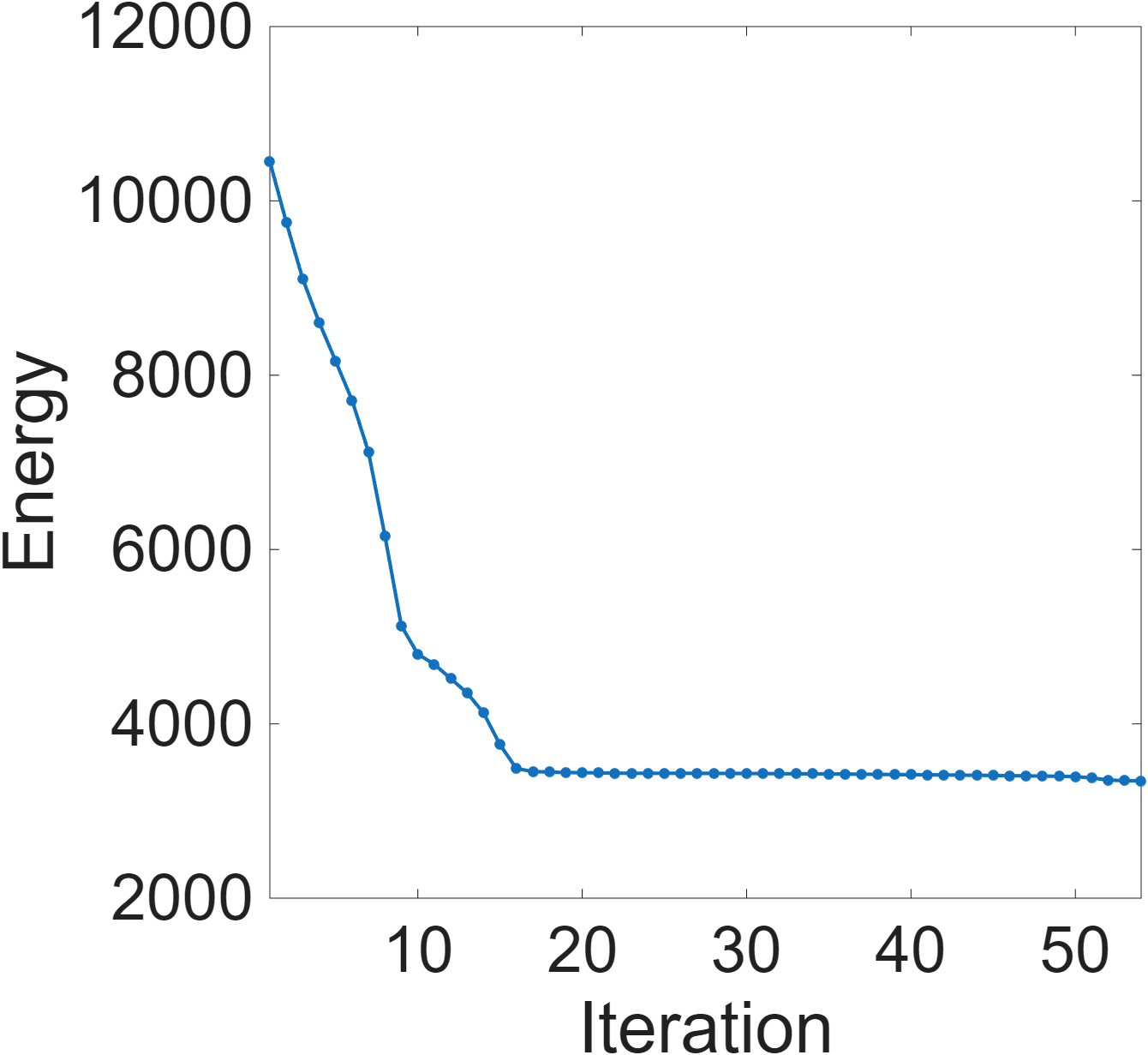}
        \caption{Energy Decay}
        \label{fig:lif_energy}
    \end{subfigure}
    
    \caption{Numerical verification of the proposed method using the LIF model. (a)--(c) The algorithm successfully overcomes intensity inhomogeneity by utilizing local intensity information to guide the contour. (d) The energy decreases monotonically, empirically confirming the unconditional stability claimed in Theorem \ref{thm:unconditional_stability} for the weighted-quantile-filter-based iteration.}
    \label{fig:lif_experiment}
\end{figure}

Figure \ref{fig:lif_experiment}(a)--(c) illustrate the robustness of the proposed method. The algorithm successfully navigates the complex energy functional to recover the accurate boundary of the object. This numerically verifies that our quantile filtering scheme effectively approximates the gradient flow of complex energy functionals.

Furthermore, the corresponding energy curve is plotted in Figure \ref{fig:lif_experiment}(d). As predicted by the stability assertion in Theorem \ref{thm:unconditional_stability}, the total energy exhibits a monotonic decay throughout the iterations, stabilizing rapidly. This result demonstrates that the proposed framework maintains its energy-dissipation property even when applied to complex functionals involving local convolutions, providing a numerical validation of the theoretical stability analysis.

\subsubsection{Comparison of Efficiency with Level Set Method}
\label{sec:efficiency_comparison}

To concretely demonstrate the computational advantages of our proposed continuous weighted-quantile-filter-based framework, we compare its efficiency against a classical level set method \cite{Chunming_Li_2008} solving the same LIF model.

Classical level set methods, which explicitly evolve the interface via gradient descent, are bounded by the Courant-Friedrichs-Lewy (CFL) stability condition, which dictates the maximum allowable time step. In contrast, our weighted-quantile-filter-based scheme guarantees unconditional energy stability and avoids the pinning effect. To ensure a fair comparison of algorithmic efficiency, we evaluate both methods using their respective appropriate time steps that maintain numerical stability, while deliberately keeping the fidelity parameter $\lambda = 3 \times 10^{-3}$ ($\tilde{\lambda} = 0.168$) and the initial conditions identical. 

The classical level set method demands 295 iterations to accurately capture the target boundary. Simultaneously, our framework converges to the same precise boundary in 47 iterations. This disparity demonstrates that despite both methods operating under their respective stable configurations, the quantile filter approach achieves faster convergence in terms of iteration count.

While this subsection focused on image segmentation, the mathematical structure of the problem is entirely consistent with interface problems in physics. In the following subsection, we demonstrate the generality of our framework by extending it to topology optimization in fluid dynamics.

\subsection{Stability Verification via Physics Constraints: Topology Optimization in Stokes Flow}
\label{subsec:stability_stokes}

To rigorously extend the validation of Theorem \ref{thm:unconditional_stability} to physics-driven problems, we apply our framework to topology optimization in fluid dynamics. Unlike the image segmentation task where the data is static, this problem involves a state equation constraint governed by the Stokes equations. This serves as a stress test for the algorithm's stability when the driving force is dynamically updated by solving a partial differential equation at each iteration.

\subsubsection{Model Formulation and Force Calculation}

We consider the optimal distribution of fluid within a design domain $\Omega$ to minimize energy dissipation. The flow is governed by the Stokes equations with a Brinkman penalization term to model the solid structures, adopting the well-established porous medium approach in fluid topology optimization \cite{Borrvall_2002, Chen_2022, Chen_2014}. The state variables velocity $\mathbf{v}$ and pressure $p$ satisfy:
\begin{equation}
    \begin{aligned}
    -\nabla \cdot (\eta D(\mathbf{v})) + \alpha(\phi) \mathbf{v} + \nabla p &= \mathbf{f} \quad \text{in } \Omega, \\
    \nabla \cdot \mathbf{v} &= 0 \quad \text{in } \Omega,
    \end{aligned}
    \label{eq:stokes_state}
\end{equation}
where $D(\mathbf{v})$ is the strain rate tensor and $\eta$ is the viscosity. The inverse permeability $\alpha(\phi)$ acts as the penalization parameter, defined as $\alpha(\phi) = \bar{\alpha}G_\tau * (1-\phi)$ to distinguish between fluid and solid regions.

Analogous to the local parameter update in the LIF model, the driving forces $F_i$ for the quantile filtering scheme are derived from the first variation of the energy functional. At each iteration $k$, given $\phi^k$, we solve \eqref{eq:stokes_state} for $\mathbf{v}^k$ and update the force field as:
\begin{equation}
    F_1^{k+1} = 0, \quad F_2^{k+1} = \frac{1}{2} \bar{\alpha} G_\tau * |\mathbf{v}^k|^2.
    \label{eq:stokes_force}
\end{equation}
These force terms $F_i$ are then substituted into the quantile filtering scheme to evolve the interface.

\subsubsection{Handling of Volume Constraint}

While the model formulation above defines the energy minimization framework, a critical step in the numerical implementation is the preservation of the fluid volume. To strictly enforce the global constraint $\int_{\Omega} \phi \, d\textbf{x} = V_{target}$ during the optimization process, we employ an augmented Lagrangian approach coupled with a root-finding algorithm. 

Mathematically, the volume constraint is incorporated into the energy functional by introducing a Lagrange multiplier $\Lambda$. The total energy becomes:
\begin{equation}
    \label{eq:energy_augmented}
    \mathcal{E}_{total}(\phi, \textbf{c}, \Lambda) = \mathcal{E}_\tau(\phi, \textbf{c}) + \Lambda \left( \int_{\Omega} \phi \, d\textbf{x} - V_{target} \right).
\end{equation}
In the context of our weighted quantile filter scheme \eqref{eq:median_update}, the inclusion of the linear volume term manifests as a global shift in the threshold level. Specifically, the updating rule is modified to:
\begin{equation}
    \label{eq:scheme_augmented}
    \phi^{k+1}(\textbf{x}) = \sup \left\{ \mu \in [0,1] : \int_{\{\textbf{y} : \phi^k(\textbf{y}) \ge \mu\}} G_\tau(\textbf{x}-\textbf{y}) \, d\textbf{y} \geq T(\textbf{x}) + \Lambda^k \right\}.
\end{equation}
where $\Lambda^k$ is a scalar offset corresponding to the Lagrange multiplier. 

In our numerical implementation, finding the optimal $\Lambda^k$ that satisfies the volume constraint is equivalent to finding the root of the monotonic function $V(\Lambda) = \int_{\Omega} \phi^{k+1}(\Lambda) \, d\textbf{x} - V_{target}$. We solve this efficiently using the bisection method at each iteration. This ensures that the generated $\phi^{k+1}$ strictly satisfies the prescribed volume fraction $\beta$, where $V_{target} = \beta |\Omega|$, with a tolerance of $10^{-5}$. With the specific handling of constraints established, we now proceed to validate the stability of the proposed framework.

While the unconditional energy stability for the unconstrained functional was established in Theorem \ref{thm:unconditional_stability}, it is crucial to ensure that the introduction of the Lagrange multiplier and the root-finding step preserve this monotonic energy decay. The following proposition rigorously justifies that our scheme remains unconditional energy stability under the strict volume constraint.

\begin{proposition}[Energy Stability under Volume Constraint]
Let $V_{target} = \int_{\Omega} \phi^0 d\mathbf{x}$. Let $\{(\phi^k, \mathbf{c}^k)\}$ be the sequence generated by the shifted thresholding rule \eqref{eq:scheme_augmented}, where $\Lambda^k$ is chosen such that $\int_{\Omega} \phi^{k+1} d\mathbf{x} = V_{target}$. Then, the sequence is unconditionally energy stable:
\begin{equation}
    \mathcal{E}_\tau(\phi^{k+1}, \mathbf{c}^{k+1}) \leq \mathcal{E}_\tau(\phi^k, \mathbf{c}^k).
\end{equation}
\end{proposition}

\begin{proof}
    See Appendix \ref{app:volume_stability}.
\end{proof}

With both the theoretical stability and the specific handling of physical constraints firmly established, we now proceed to numerically validate these properties.

\subsubsection{Validation of Stability: Energy Decay}

The validation of the unconditional stability property focuses on the monotonicity of the energy evolution for these physics constrained problems. To demonstrate the universality and robustness of our framework, we conduct stress tests on two distinct benchmark examples with different Dirichlet boundary conditions:

\begin{enumerate}
    \item A ``Flow Contraction'' problem. As shown in Figure \ref{fig:stokes_contraction}(a), fluid enters through the entire left boundary with a height of $1$ and is constrained to exit through a narrow outlet of height $1/3$ located at the center of the right boundary. This setup imposes a severe geometric constraint, challenging the algorithm to form a smooth converging channel while minimizing energy dissipation.
    
    \item A ``Two-Inlet Two-Outlet'' problem. As illustrated in Figure \ref{fig:stokes_double_pipe}(a), the domain features two separate inlets on the left boundary and two separate outlets on the right. Each opening has a height of $1/6$, symmetrically distributed. This configuration serves as a stress test to verify whether the algorithm will preserve the topological separation of the flow paths or merge them to minimize resistance, depending on the energy functional.
\end{enumerate}

\begin{figure}[htbp]
    \centering
    \begin{subfigure}[b]{0.29\textwidth}
        \centering
        \includegraphics[width=\textwidth]{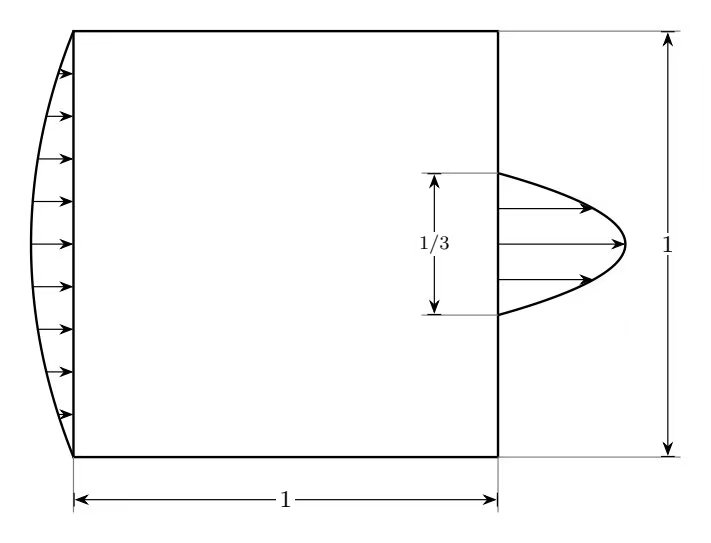} 
        \vspace{-1em}
        \caption{Wide $\to$ Narrow}
        \label{fig:cont_setup}
    \end{subfigure}
    \hfill
    \begin{subfigure}[b]{0.21\textwidth}
        \centering
        \includegraphics[width=\textwidth]{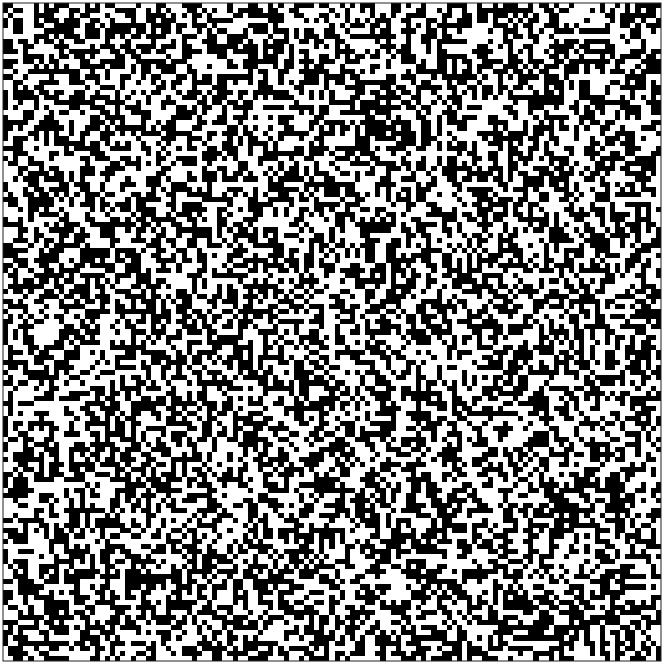}
        \vspace{-0.5em}
        \caption{Initial}
        \label{fig:cont_init}
    \end{subfigure}
    \hfill
    \begin{subfigure}[b]{0.21\textwidth}
        \centering
        \includegraphics[width=\textwidth]{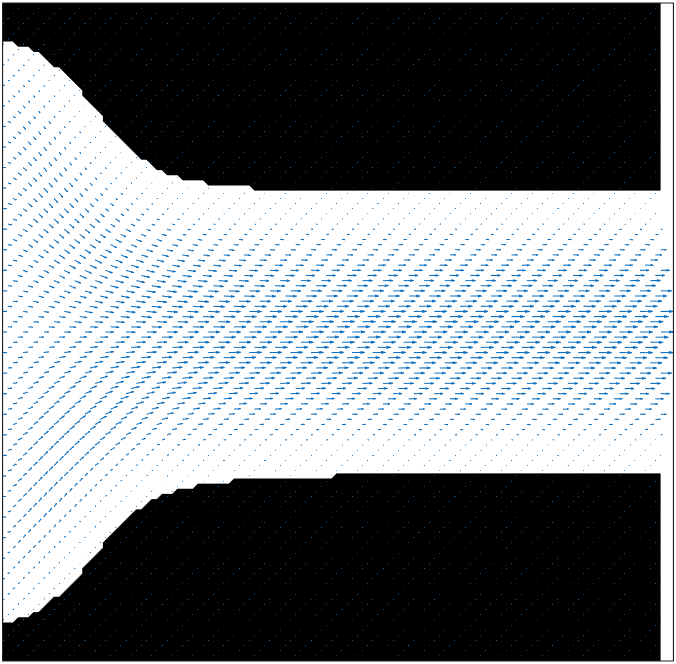}
        \vspace{-0.5em}
        \caption{Final}
        \label{fig:cont_final}
    \end{subfigure}
    \hfill
    \begin{subfigure}[b]{0.23\textwidth}
        \centering
        \includegraphics[width=\textwidth]{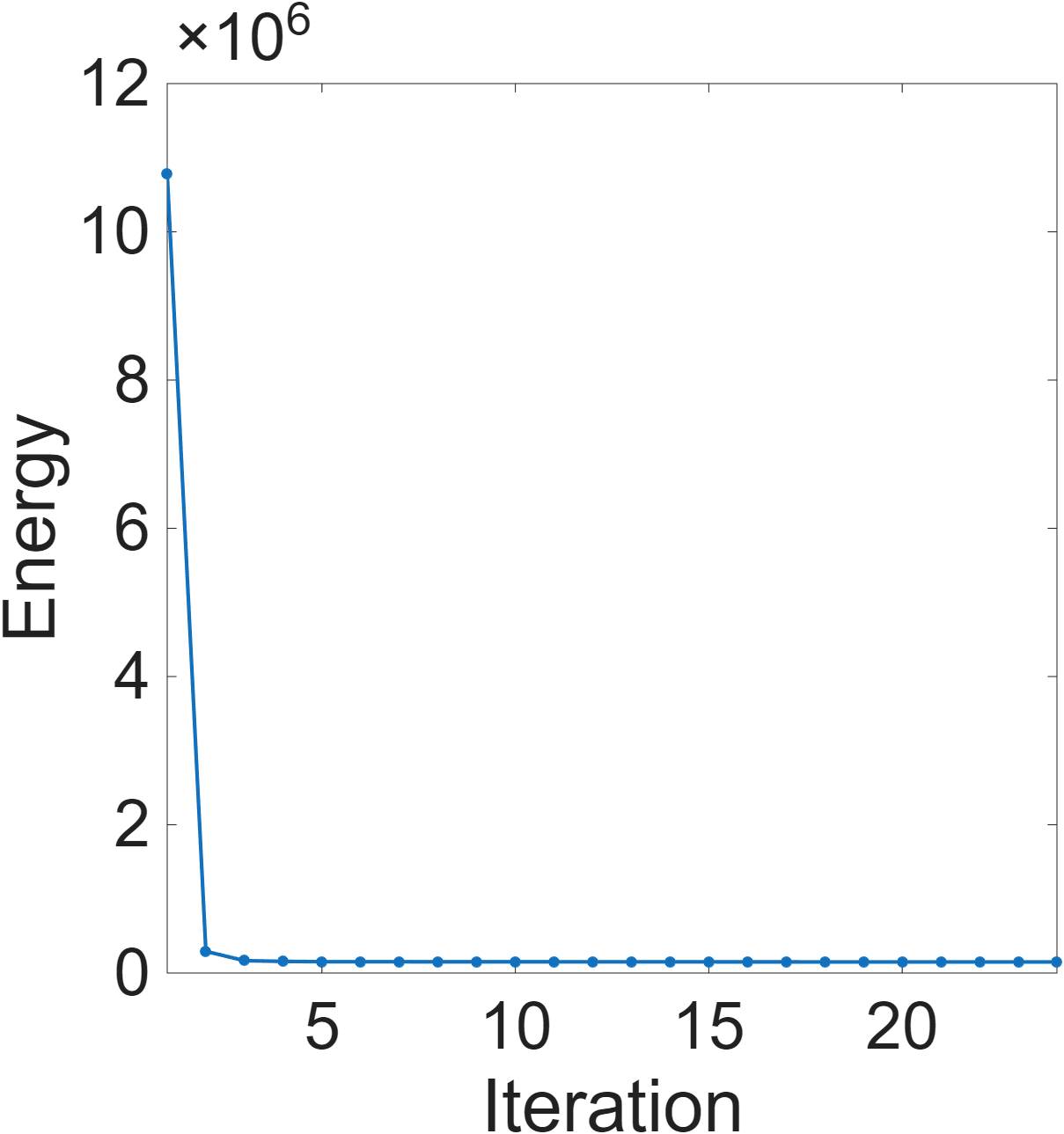}
        \caption{Energy Decay}
        \label{fig:cont_energy}
    \end{subfigure}
    
    \caption{Numerical verification of the proposed framework on the Flow Contraction problem. (a) Domain setup with a full-height inlet transitioning to a restricted outlet of height $1/3$. (b) Initial random state. (c) The algorithm successfully forms a smooth converging channel. (d) The energy curve confirms unconditional stability.}
    \label{fig:stokes_contraction}
\end{figure}

\begin{figure}[htbp]
    \centering
    \begin{subfigure}[t]{0.34\textwidth}
        \centering
        \includegraphics[width=\textwidth]{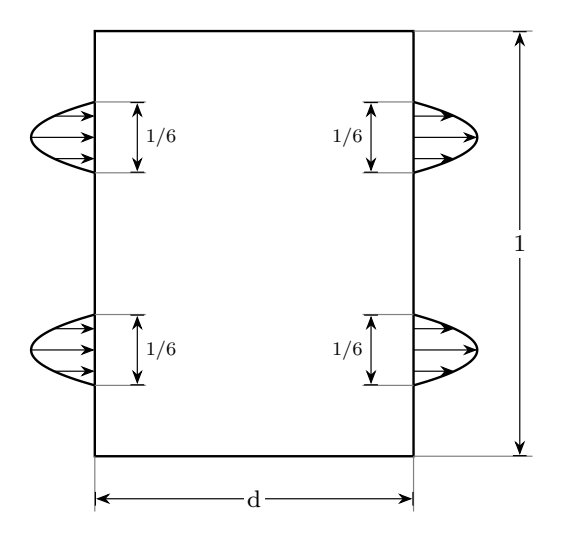} 
        \caption{2 Inlets $\to$ 2 Outlets}
        \vspace{1mm}
        \label{fig:exp1_bc}
    \end{subfigure}
    \hfill
    \begin{subfigure}[t]{0.16\textwidth}
        \centering
        \includegraphics[width=\textwidth]{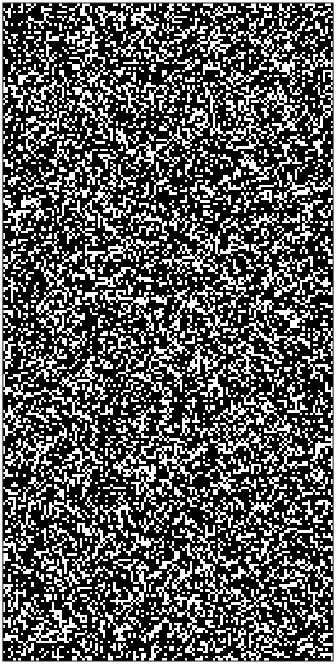}
        \caption{Initial}
        \label{fig:exp1_init}
    \end{subfigure}
    \hfill
    \begin{subfigure}[t]{0.16\textwidth}
        \centering
        \includegraphics[width=\textwidth]{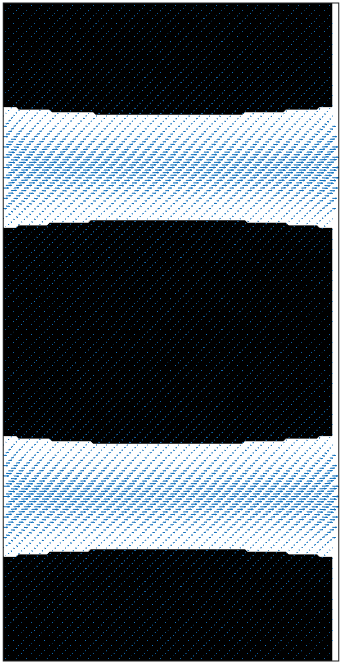}
        \caption{Final}
        \label{fig:exp1_final}
    \end{subfigure}
    \hfill
    \begin{subfigure}[t]{0.30\textwidth}
        \centering
        \includegraphics[width=\textwidth]{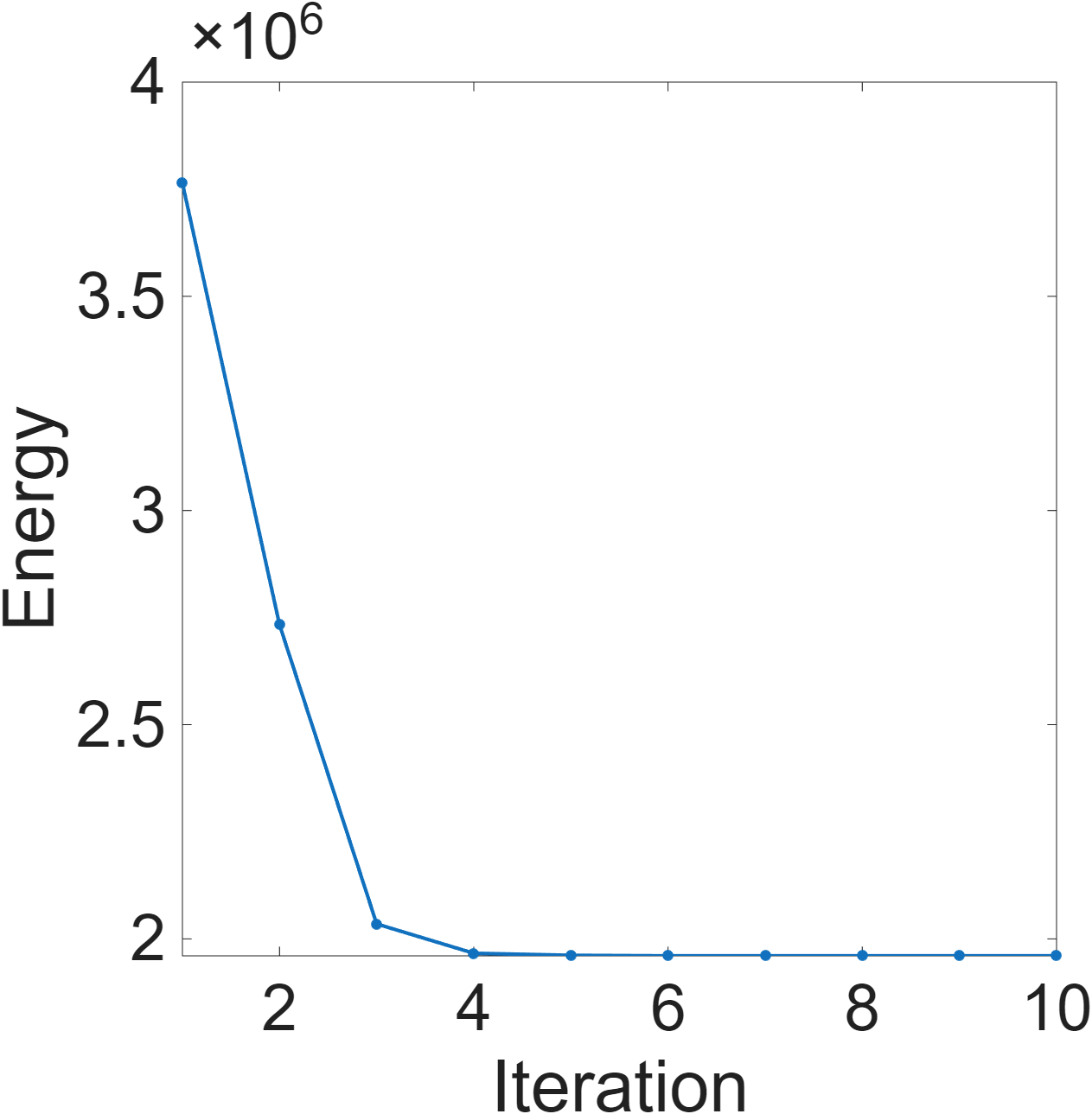}
        \caption{Energy Decay}
        \label{fig:exp1_energy}
    \end{subfigure}
    
    \caption{Numerical verification of the proposed framework on the ``Two-Inlet Two-Outlet'' problem. (a) Domain setup featuring two separate inlets and outlets, each with a height of $1/6$. (b) Initial random state. (c) The energy curve confirms unconditional stability. (d) The algorithm successfully recovers a double-channel structure.}
    \label{fig:stokes_double_pipe}
\end{figure}

Figures \ref{fig:stokes_contraction} and \ref{fig:stokes_double_pipe} present the optimization processes for the ``Flow Contraction'' problem and the ``Two-Inlet Two-Outlet'' problem, respectively. To test the robustness of the algorithm against local minima, we initialize both designs with random noise, representing randomized initial material distributions.

The visualization of the results illustrates the capability of the proposed method to self-organize from a random state in different geometric settings. In both cases, the algorithm successfully navigates the complex energy functional, where the ``force'' is implicitly determined by the PDE solution, to recover a smooth channel structure. This numerically verifies that our quantile filtering scheme effectively approximates the gradient flow even when the energy functional involves global PDE constraints.

Furthermore, the corresponding energy curves are plotted in Figures \ref{fig:stokes_contraction}(c) and \ref{fig:stokes_double_pipe}(c). As predicted by the stability assertion in Theorem \ref{thm:unconditional_stability}, the total energy dissipation exhibits a monotonic decay throughout the iterations for both examples, stabilizing rapidly. These results collaboratively demonstrate that the proposed framework maintains its energy-dissipation property across different physics-driven applications, providing strong numerical validation of the theoretical stability analysis.

\section{Conclusions}
\label{sec:conclusion}

In this paper, we have proposed a general and robust numerical framework based on the quantile filter scheme for solving variational problems with perimeter regularization. By interpreting the quantile filter as a distinct threshold dynamics solver, we established a rigorous theoretical foundation for the method. Specifically, we proved the unconditional energy stability of the iterative scheme and demonstrated that, for generic data, the proposed convex relaxation naturally enforces a binary solution without the need for explicit penalization terms. A key advantage of our approach is its capability to effectively mitigate the pinning effect, which commonly plagues traditional threshold dynamics methods, without sacrificing computational efficiency.

Extensive numerical experiments have verified the versatility and effectiveness of the proposed framework. We successfully applied the method to two distinct fields: image segmentation (using Chan-Vese and local intensity fitting models) and topology optimization for Stokes flow. The results indicate that our algorithm not only accurately captures complex topological changes and preserves volume constraints but also maintains stability across a wide range of parameters. The seamless integration of the level-set formulation with the quantile filtering operation provides a simple yet powerful tool for interface related optimization problems.

While our results are encouraging, several avenues for future research remain open. Specifically, the current framework is designed primarily for binary-phase problems; extending the weighted quantile filter based approach to multiphase problems with multiple junctions remains a challenging but worthwhile direction \cite{Laux_2018}. Beyond the geometric complexity of tracking triple junctions, extending our method to ternary or multiphase physical systems, such as the ternary Ohta-Kawasaki model, poses additional challenges due to the complexity of the underlying physical laws, particularly the non-local long-range interaction terms \cite{Choi_2022, Luo_2024, Xu_2019}. Furthermore, although we have addressed topology optimization in Stokes flow, applying this method to Navier-Stokes equations or fluid-structure interaction problems would require further investigation into the coupling strategies \cite{Feppon_2020}.

\appendix

\section{Proof of Existence and Uniqueness}
\label{app:existence_uniqueness}

To ensure the proof of Lemma \ref{lem:existence_uniqueness} is fully self-contained, we first explicitly state the foundational definitions and theorems from functional analysis that are invoked.

\subsection{Preliminaries}

We begin with the geometric concept of dentability and its relationship to compact sets in a Banach space.

\begin{definition}[Dentable Set \cite{rieffel1967dentable}]
A subset $\mathcal{K}$ of a Banach space is called dentable if for every $\epsilon > 0$ there is a $b \in \mathcal{K}$ such that $b \notin \overline{c}(\mathcal{K} - \text{ball}(b, \epsilon))$, where $\overline{c}$ denotes the closed convex hull.
\end{definition}

\begin{proposition}[Rieffel \cite{rieffel1967dentable}] \label{prop:rieffel}
Any relatively norm compact subset of a Banach space is dentable.
\end{proposition}

Next, we formally define a strong minimizer, which is central to establishing the uniqueness of the solution, as well as the concept of lower semicontinuity.

\begin{definition}[Strong Minimizer \cite{Loewen_2001}]
Let $X$ be a Banach space. For a function $f: X \to \mathbb{R} \cup \{+\infty\}$, a point $x_0 \in X$ is a strong minimizer if $f(x_0) = \inf_X f$ and every sequence $(x_n)$ along which $f(x_n) \to \inf_X f$ obeys $\|x_n - x_0\| \to 0$.
\end{definition}

\begin{proposition}[Strict Minimality \cite{Loewen_2001}] \label{prop:strict_min}
A strong minimizer is necessarily a strict minimizer, that is, $f(x) > f(x_0)$ for every $x \neq x_0$. Consequently, if a strong minimizer exists, it is unique.
\end{proposition}

\begin{definition}[Lower Semicontinuity \cite{Lieb_2001}]
A functional $f: X \to \mathbb{R} \cup \{+\infty\}$ on a metric space $X$ is lower semicontinuous at a point $x_0 \in X$ if for every sequence $x_n \to x_0$, it holds that $\liminf_{n \to \infty} f(x_n) \ge f(x_0)$. It is lower semicontinuous on $X$ if this inequality holds for all points in $X$.
\end{definition}

Finally, we state Stegall's Variational Principle, which connects the geometric property of dentability with the existence of strong minimizers under linear perturbations.

\begin{theorem}[Stegall's Variational Principle \cite{Lassonde_2009, stegall1978optimization}] \label{thm:stegall}
Every nonempty subset of a closed bounded convex set $\mathcal{C}$ in a Banach space is dentable if and only if for every lower-bounded lower semicontinuous function $J: \mathcal{C} \to \mathbb{R}$ there exists a dense $G_\delta$ subset $\mathcal{G}$ of the dual space $X^*$ such that for each $h \in \mathcal{G}$, the perturbed functional $J + h$ attains a strong minimum.
\end{theorem}

\subsection{Proof of Lemma \ref{lem:existence_uniqueness}}

We now restate and rigorously prove the existence and uniqueness lemma.

\begin{lemma}[Existence and Uniqueness] 
Let $X$ be a Banach space, and let $\mathcal{K} \subset X$ be a compact, convex subset. Let $J: \mathcal{K} \to \mathbb{R}$ be a lower-bounded lower semicontinuous functional. Consider the perturbed functional $\mathcal{E}_{\tau}(\phi) = J(\phi) + \langle h, \phi \rangle$ defined on $\mathcal{K}$. Then there exists a dense $G_\delta$ subset $\mathcal{G}$ of $X^*$ such that for every perturbation $h \in \mathcal{G}$, the functional $\mathcal{E}_{\tau}$ admits a unique global minimizer $\phi^*$ in $\mathcal{K}$.
\end{lemma}

\begin{proof}
Since $\mathcal{K}$ is a compact convex subset of the Banach space $X$. It follows directly from Proposition \ref{prop:rieffel} that $\mathcal{K}$ is dentable.

Because $\mathcal{K}$ is a closed, bounded, convex set whose subsets are all dentable, and the unperturbed functional $J$ is lower-bounded and lower semicontinuous on $\mathcal{K}$, the conditions for Theorem \ref{thm:stegall} are perfectly satisfied.

Invoking Theorem \ref{thm:stegall}, there exists a dense $G_\delta$ set $\mathcal{G}$ in the dual space $X^*$ such that for all continuous linear perturbations $h \in \mathcal{G}$, the perturbed functional $\mathcal{E}_{\tau}(\phi) = J(\phi) + \langle h, \phi \rangle$ attains a strong minimum on $\mathcal{K}$.

Finally, by Proposition \ref{prop:strict_min}, any strong minimizer is strictly minimal and thus unique. Therefore, for every $h \in \mathcal{G}$, the global minimizer $\phi^*$ of the functional $\mathcal{E}_{\tau}$ is strictly unique, completing the proof.
\end{proof}

\section{Proof of Energy Stability under Volume Constraint}
\label{app:volume_stability}

In this appendix, we provide the rigorous mathematical proof that the unconditional energy stability of our quantile filter framework is preserved under the global volume constraint, as discussed in Section \ref{subsec:stability_stokes}.

\begin{proposition}[Energy Stability under Volume Constraint]
\label{prop:volume_stability}
Let $V_{target} = \int_{\Omega} \phi^0 d\mathbf{x}$. Let $\{(\phi^k, \mathbf{c}^k)\}$ be the sequence generated by the shifted thresholding rule \eqref{eq:scheme_augmented}, where $\Lambda^k$ is chosen such that $\int_{\Omega} \phi^{k+1} d\mathbf{x} = V_{target}$. Then, the sequence is unconditionally energy stable:
\begin{equation}
    \mathcal{E}_\tau(\phi^{k+1}, \mathbf{c}^{k+1}) \leq \mathcal{E}_\tau(\phi^k, \mathbf{c}^k).
\end{equation}
\end{proposition}

\begin{proof}
Define the volume-constrained admissible set $\mathcal{V}_{ad} = \{ \phi \in BV(\Omega; [0,1]) : \int_{\Omega} \phi d\mathbf{x} = V_{target} \}$. At iteration $k+1$ (with fixed $\mathbf{c}^{k+1}$), we aim to minimize the alternative energy over $\mathcal{V}_{ad}$. 

Introducing the Lagrange multiplier $\Lambda$, the unconstrained Lagrangian is:
\begin{equation}
    \mathcal{L}_k(\phi, \Lambda) = \mathcal{E}_\tau(\phi, \mathbf{c}^{k+1}) + \mathcal{M}_k(\phi) + \Lambda \left( \int_{\Omega} \phi d\mathbf{x} - V_{target} \right).
\end{equation}

Because $\mathcal{L}_k(\phi, \Lambda)$ is linear in $\phi$ without spatial derivatives, its minimization decouples pointwise. As derived in Section \ref{subsubsec:derivation}, the pointwise minimizer $\phi_\Lambda$ for a fixed $\Lambda$ is exactly given by \eqref{eq:scheme_augmented}.

Our bisection root-finding step guarantees an optimal multiplier $\Lambda^k$ such that the minimizer $\phi^{k+1} := \phi_{\Lambda^k}$ exactly satisfies the volume constraint, ensuring $\phi^{k+1} \in \mathcal{V}_{ad}$.

Since $\phi^{k+1}$ globally minimizes $\mathcal{L}_k(\cdot, \Lambda^k)$, for any competing state $\phi \in \mathcal{V}_{ad}$, we have:
\begin{equation}
    \mathcal{L}_k(\phi^{k+1}, \Lambda^k) \leq \mathcal{L}_k(\phi, \Lambda^k).
\end{equation}

Expanding this Lagrangian and noting that the penalty term $\Lambda^k ( \int_{\Omega} \phi d\mathbf{x} - V_{target} )$ strictly vanishes for any function in $\mathcal{V}_{ad}$, the inequality simplifies directly to:
\begin{equation}
    \mathcal{E}_\tau(\phi^{k+1}, \mathbf{c}^{k+1}) + \mathcal{M}_k(\phi^{k+1}) \leq \mathcal{E}_\tau(\phi, \mathbf{c}^{k+1}) + \mathcal{M}_k(\phi), \quad \forall \phi \in \mathcal{V}_{ad}.
\end{equation}
This establishes that $\phi^{k+1}$ is the exact global minimizer of the alternative energy over $\mathcal{V}_{ad}$. 

Finally, since the previous state $\phi^k \in \mathcal{V}_{ad}$ is a valid test function, we conclude with the following inequalities:
\begin{equation}
\begin{aligned}
\mathcal{E}_\tau(\phi^{k+1}, \mathbf{c}^{k+1}) &\leq \mathcal{E}_\tau(\phi^{k+1}, \mathbf{c}^{k+1}) + \mathcal{M}_k(\phi^{k+1}) \quad &(\text{Non-negativity of } \mathcal{M}_k) \\
&\leq \mathcal{E}_\tau(\phi^{k}, \mathbf{c}^{k+1}) + \mathcal{M}_k(\phi^{k}) \quad &(\text{Minimality of } \phi^{k+1} \text{ over } \mathcal{V}_{ad}) \\
&= \mathcal{E}_\tau(\phi^{k}, \mathbf{c}^{k+1}) \quad &(\text{Contact property } \mathcal{M}_k(\phi^k) = 0) \\
&\leq \mathcal{E}_\tau(\phi^{k}, \mathbf{c}^{k}). \quad &(\text{Minimality of } \mathbf{c}^{k+1})
\end{aligned}
\end{equation}
This confirms the strict energy dissipation of the volume-constrained scheme.
\end{proof}

\bibliographystyle{siamplain}
\bibliography{references}
\end{document}

%% file: ex_shared.tex
\usepackage{lipsum}
\usepackage{amsfonts}
\usepackage{graphicx}
\usepackage{epstopdf}
\usepackage{algorithmic}
\ifpdf
  \DeclareGraphicsExtensions{.eps,.pdf,.png,.jpg}
\else
  \DeclareGraphicsExtensions{.eps}
\fi

\newsiamremark{remark}{Remark}
\newsiamremark{hypothesis}{Hypothesis}
\crefname{hypothesis}{Hypothesis}{Hypotheses}
\newsiamthm{claim}{Claim}
\newsiamremark{fact}{Fact}
\crefname{fact}{Fact}{Facts}
\usepackage{lineno}
\usepackage{cite}

\headers{weighted quantile filter based framework}{Sihao Cheng, Ziming Shao, and Dong Wang}

\title{A weighted quantile filter based framework for interface optimal design problems\thanks{This work is partially supported by the National Natural Science Foundation of China (Grant No. 12422116), Shenzhen Science and Technology Innovation Program (Grant Nos. QNXMA20250701091827037, RCYX20221008092843046),
and Hetao Shenzhen-Hong Kong Science and Technology Innovation Cooperation Zone Project (No. HZQSWS-KCCYB-2024016).}}

\author{Sihao Cheng\thanks{School of Science and Engineering, The Chinese University of Hong Kong (Shenzhen), Shenzhen, Guangdong 518172, China (\email{sihaocheng@link.cuhk.edu.cn}).}
\and Ziming Shao\thanks{School of Science and Engineering, The Chinese University of Hong Kong (Shenzhen), Shenzhen, Guangdong 518172, China (\email{zimingshao@link.cuhk.edu.cn}).}
\and Dong Wang\thanks{School of Science and Engineering, The Chinese University of Hong Kong (Shenzhen), Shenzhen, Guangdong 518172, China \& Shenzhen International Center for Industrial and Applied Mathematics, Shenzhen Research Institute of Big Data, Guangdong 518172, China \& Shenzhen Loop Area Institute, Guangdong 518048, China (\email{wangdong@cuhk.edu.cn}).}}

\usepackage{amsopn}


%% file: references.bib
@article{Allaire_2004, title={Structural optimization using sensitivity analysis and a level-set method}, volume={194}, ISSN={0021-9991}, DOI={10.1016/j.jcp.2003.09.032}, number={1}, journal={Journal of Computational Physics}, publisher={Elsevier BV}, author={Allaire, Grégoire and Jouve, François and Toader, Anca-Maria}, year={2004}, month=Feb, pages={363–393} }

@inbook{Ambrosio_2000, title={Functions of Bounded Variation}, ISBN={9781383020311}, DOI={10.1093/oso/9780198502456.003.0003}, booktitle={Functions of Bounded Variation and Free Discontinuity Problems}, publisher={Oxford University PressOxford}, author={Ambrosio, Luigi and Fusco, Nicola and Pallara, Diego}, year={2000}, month=Mar, pages={116–210} }

@article{Bendsoe_1989, title={Optimal shape design as a material distribution problem}, volume={1}, ISSN={1615-1488}, DOI={10.1007/bf01650949}, number={4}, journal={Structural Optimization}, publisher={Springer Science and Business Media LLC}, author={Bendsøe, M. P.}, year={1989}, month=Dec, pages={193–202} }

@article{Borrvall_2002, title={Topology optimization of fluids in Stokes flow}, volume={41}, ISSN={1097-0363}, DOI={10.1002/fld.426}, number={1}, journal={International Journal for Numerical Methods in Fluids}, publisher={Wiley}, author={Borrvall, Thomas and Petersson, Joakim}, year={2002}, month=Dec, pages={77–107} }

@article{Chan_2001, title={Active contours without edges}, volume={10}, ISSN={1057-7149}, DOI={10.1109/83.902291}, number={2}, journal={IEEE Transactions on Image Processing}, publisher={Institute of Electrical and Electronics Engineers (IEEE)}, author={Chan, T.F. and Vese, L.A.}, year={2001}, pages={266–277} }

@article{Chan_2006, title={Algorithms for Finding Global Minimizers of  Image Segmentation and Denoising Models}, volume={66}, ISSN={1095-712X}, DOI={10.1137/040615286}, number={5}, journal={SIAM Journal on Applied Mathematics}, publisher={Society for Industrial & Applied Mathematics (SIAM)}, author={Chan, Tony F. and Esedoglu, Selim and Nikolova, Mila}, year={2006}, month=Jan, pages={1632–1648} }

@article{Chen_2024, title={A prediction-correction based iterative convolution-thresholding method for topology optimization of heat transfer problems}, volume={511}, ISSN={0021-9991}, DOI={10.1016/j.jcp.2024.113119}, journal={Journal of Computational Physics}, publisher={Elsevier BV}, author={Chen, Huangxin and Dong, Piaopiao and Wang, Dong and Wang, Xiao-Ping}, year={2024}, month=Aug, pages={113119} }

@article{Chen_2026, title={A robust and stable phase field method for structural topology optimization}, volume={547}, ISSN={0021-9991}, DOI={10.1016/j.jcp.2025.114531}, journal={Journal of Computational Physics}, publisher={Elsevier BV}, author={Chen, Huangxin and Dong, Piaopiao and Wang, Dong and Wang, Xiao-Ping}, year={2026}, month=Feb, pages={114531} }

@article{Chen_2022, title={An Efficient Threshold Dynamics Method for Topology Optimization for Fluids}, volume={3}, ISSN={2708-0579}, DOI={10.4208/csiam-am.so-2021-0007}, number={1}, journal={CSIAM Transactions on Applied Mathematics}, publisher={Global Science Press}, author={Chen, Huangxin and Leng, Haitao and Wang, Dong and Wang, Xiao-Ping}, year={2022}, month=Jun, pages={26–56} }

@article{Chen_2014, title={A one-domain approach for modeling and simulation of free fluid over a porous medium}, volume={259}, ISSN={0021-9991}, DOI={10.1016/j.jcp.2013.12.008}, journal={Journal of Computational Physics}, publisher={Elsevier BV}, author={Chen, Huangxin and Wang, Xiao-Ping}, year={2014}, month=Feb, pages={650–671} }

@article{Choi_2022, title={Second-order stabilized semi-implicit energy stable schemes for bubble assemblies in binary and ternary systems}, volume={27}, ISSN={1553-524X}, DOI={10.3934/dcdsb.2021246}, number={8}, journal={Discrete and Continuous Dynamical Systems - B}, publisher={American Institute of Mathematical Sciences (AIMS)}, author={Choi, Hyunjung and Zhao, Yanxiang}, year={2022}, pages={4649} }

@article{Deckelnick_2005, title={Computation of geometric partial differential equations and mean curvature flow}, volume={14}, ISSN={1474-0508}, DOI={10.1017/s0962492904000224}, journal={Acta Numerica}, publisher={Cambridge University Press (CUP)}, author={Deckelnick, Klaus and Dziuk, Gerhard and Elliott, Charles M.}, year={2005}, month=Apr, pages={139–232} }

@inbook{Du_2020, title={The phase field method for geometric moving interfaces and their numerical approximations}, ISBN={9780444640031}, ISSN={1570-8659}, DOI={10.1016/bs.hna.2019.05.001}, booktitle={Geometric Partial Differential Equations - Part I}, publisher={Elsevier}, author={Qiang Du and Xiaobing Feng}, year={2020}, pages={425–508} }

@article{Esedog_lu_2010, title={Diffusion generated motion using signed distance functions}, volume={229}, ISSN={0021-9991}, DOI={10.1016/j.jcp.2009.10.002}, number={4}, journal={Journal of Computational Physics}, publisher={Elsevier BV}, author={Esedoḡlu, Selim and Ruuth, Steven and Tsai, Richard}, year={2010}, month=Feb, pages={1017–1042} }

@article{Esedo_Lu_2014, title={Threshold Dynamics for Networks with Arbitrary Surface Tensions}, volume={68}, ISSN={1097-0312}, DOI={10.1002/cpa.21527}, number={5}, journal={Communications on Pure and Applied Mathematics}, publisher={Wiley}, author={Esedoḡlu, Selim and Otto, Felix}, year={2014}, month=Jun, pages={808–864} }

@article{Esedo_lu_2024, title={On median filters for motion by mean curvature}, ISSN={1088-6842}, DOI={10.1090/mcom/3940}, journal={Mathematics of Computation}, publisher={American Mathematical Society (AMS)}, author={Esedoḡlu, Selim and Guo, Jiajia and Li, David}, year={2024}, month=Feb }

@article{Feppon_2020, title={Topology optimization of thermal fluid–structure systems using body-fitted meshes and parallel computing}, volume={417}, ISSN={0021-9991}, DOI={10.1016/j.jcp.2020.109574}, journal={Journal of Computational Physics}, publisher={Elsevier BV}, author={Feppon, F. and Allaire, G. and Dapogny, C. and Jolivet, P.}, year={2020}, month=Sep, pages={109574} }

@article{Guo_2025, title={Median Filters for Anisotropic Wetting / Dewetting Problems}, volume={47}, ISSN={1095-7197}, DOI={10.1137/24m1670755}, number={3}, journal={SIAM Journal on Scientific Computing}, publisher={Society for Industrial & Applied Mathematics (SIAM)}, author={Guo, Jiajia and Esedoḡlu, Selim}, year={2025}, month=Jun, pages={A2012–A2039} }

@article{Hartmann_2010, title={The constrained reinitialization equation for level set methods}, volume={229}, ISSN={0021-9991}, DOI={10.1016/j.jcp.2009.10.042}, number={5}, journal={Journal of Computational Physics}, publisher={Elsevier BV}, author={Hartmann, Daniel and Meinke, Matthias and Schröder, Wolfgang}, year={2010}, month=Mar, pages={1514–1535} }

@article{Hirt_1981, title={Volume of fluid (VOF) method for the dynamics of free boundaries}, volume={39}, ISSN={0021-9991}, DOI={10.1016/0021-9991(81)90145-5}, number={1}, journal={Journal of Computational Physics}, publisher={Elsevier BV}, author={Hirt, C.W and Nichols, B.D}, year={1981}, month=Jan, pages={201–225} }

@article{Jacqmin_1999, title={Calculation of Two-Phase Navier–Stokes Flows Using Phase-Field Modeling}, volume={155}, ISSN={0021-9991}, DOI={10.1006/jcph.1999.6332}, number={1}, journal={Journal of Computational Physics}, publisher={Elsevier BV}, author={Jacqmin, David}, year={1999}, month=Oct, pages={96–127} }

@article{Lassonde_2009, title={Asplund Spaces, Stegall Variational Principle and the RNP}, volume={17}, ISSN={1877-0541}, DOI={10.1007/s11228-009-0111-6}, number={2}, journal={Set-Valued and Variational Analysis}, publisher={Springer Science and Business Media LLC}, author={Lassonde, Marc}, year={2009}, month=May, pages={183–193} }

@article{Laux_2018, title={Convergence of the Allen‐Cahn Equation to Multiphase Mean Curvature Flow}, volume={71}, ISSN={1097-0312}, DOI={10.1002/cpa.21747}, number={8}, journal={Communications on Pure and Applied Mathematics}, publisher={Wiley}, author={Laux, Tim and Simon, Theresa M.}, year={2018}, month=Apr, pages={1597–1647} }

@article{Chunming_Li_2008, title={Minimization of Region-Scalable Fitting Energy for Image Segmentation}, volume={17}, ISSN={1057-7149}, DOI={10.1109/tip.2008.2002304}, number={10}, journal={IEEE Transactions on Image Processing}, publisher={Institute of Electrical and Electronics Engineers (IEEE)}, author={Chunming Li and Chiu-Yen Kao and Gore, J.C. and Zhaohua Ding}, year={2008}, month=Oct, pages={1940–1949} }

@inproceedings{Li_2007, title={Implicit Active Contours Driven by Local Binary Fitting Energy}, DOI={10.1109/cvpr.2007.383014}, booktitle={2007 IEEE Conference on Computer Vision and Pattern Recognition}, publisher={IEEE}, author={Li, Chunming and Kao, Chiu-Yen and Gore, John C. and Ding, Zhaohua}, year={2007}, month=Jun, pages={1–7} }

@article{Li_2026, title={High-Order Unconditionally Energy-Stable Decoupled Method for Phase Field Modeling of Pitting Corrosion with Adaptive Implementation}, volume={18}, ISSN={2070-0733}, DOI={10.4208/aamm.oa-2025-0232}, number={5}, journal={Advances in Applied Mathematics and Mechanics}, publisher={Global Science Press}, author={Li, Futuan and Li, Hongwei and Tang, Tao and Yang, Jiang}, year={2026}, month=Jan, pages={1355–1388} }

@book{Lieb_2001, title={Analysis}, ISBN={9781470411435}, ISSN={1065-7339}, DOI={10.1090/gsm/014}, journal={Graduate Studies in Mathematics}, publisher={American Mathematical Society}, author={Lieb, Elliott and Loss, Michael}, year={2001}, month=Mar }

@article{Liu_2022, title={Two-Phase Segmentation for Intensity Inhomogeneous Images by the Allen--Cahn Local Binary Fitting Model}, volume={44}, ISSN={1095-7197}, DOI={10.1137/21m1421830}, number={1}, journal={SIAM Journal on Scientific Computing}, publisher={Society for Industrial & Applied Mathematics (SIAM)}, author={Liu, Chaoyu and Qiao, Zhonghua and Zhang, Qian}, year={2022}, month=Jan, pages={B177–B196} }

@article{Liu_2023, title={An Active Contour Model with Local Variance Force Term and Its Efficient Minimization Solver for Multiphase Image Segmentation}, volume={16}, ISSN={1936-4954}, DOI={10.1137/22m1483645}, number={1}, journal={SIAM Journal on Imaging Sciences}, publisher={Society for Industrial & Applied Mathematics (SIAM)}, author={Liu, Chaoyu and Qiao, Zhonghua and Zhang, Qian}, year={2023}, month=Feb, pages={144–168} }

@article{Loewen_2001, title={A Generalized Variational Principle}, volume={53}, ISSN={1496-4279}, DOI={10.4153/cjm-2001-044-8}, number={6}, journal={Canadian Journal of Mathematics}, publisher={Canadian Mathematical Society}, author={Loewen, Philip D. and Wang, Xianfu}, year={2001}, month=Dec, pages={1174–1193} }

@article{Luo_2024, title={Asymptotically compatible schemes for nonlocal Ohta–Kawasaki model}, volume={40}, ISSN={1098-2426}, DOI={10.1002/num.23143}, number={6}, journal={Numerical Methods for Partial Differential Equations}, publisher={Wiley}, author={Luo, Wangbo and Zhao, Yanxiang}, year={2024}, month=Aug }

@article{Ma_2021, title={A Characteristic Function-Based Algorithm for Geodesic Active Contours}, volume={14}, ISSN={1936-4954}, DOI={10.1137/20m1382817}, number={3}, journal={SIAM Journal on Imaging Sciences}, publisher={Society for Industrial & Applied Mathematics (SIAM)}, author={Ma, Jun and Wang, Dong and Wang, Xiao-Ping and Yang, Xiaoping}, year={2021}, month=Jan, pages={1184–1205} }

@article{Merriman_1994, title={Motion of Multiple Junctions: A Level Set Approach}, volume={112}, ISSN={0021-9991}, DOI={10.1006/jcph.1994.1105}, number={2}, journal={Journal of Computational Physics}, publisher={Elsevier BV}, author={Merriman, Barry and Bence, James K. and Osher, Stanley J.}, year={1994}, month=Jun, pages={334–363} }

@article{miranda2007short,
  title = {Short-Time Heat Flow and Functions of Bounded Variation in {{R N}}},
  author = {Miranda, Michele and Pallara, Diego and Paronetto, Fabio and Preunkert, Marc},
  year = 2008,
  month = dec,
  journal = {Annales de la Facult\'e des sciences de Toulouse : Math\'ematiques},
  volume = {16},
  number = {1},
  pages = {125--145},
  issn = {2258-7519},
  doi = {10.5802/afst.1142},
  urldate = {2026-02-04},
  langid = {english}
}

@article{Mumford_1989, title={Optimal approximations by piecewise smooth functions and associated variational problems}, volume={42}, ISSN={1097-0312}, DOI={10.1002/cpa.3160420503}, number={5}, journal={Communications on Pure and Applied Mathematics}, publisher={Wiley}, author={Mumford, David and Shah, Jayant}, year={1989}, month=Jul, pages={577–685} }

@article{Oberman_2004, title={A convergent monotone difference scheme for motion of level sets by mean curvature}, volume={99}, ISSN={0945-3245}, DOI={10.1007/s00211-004-0566-1}, number={2}, journal={Numerische Mathematik}, publisher={Springer Science and Business Media LLC}, author={Oberman, Adam M.}, year={2004}, month=Nov, pages={365–379} }

@article{Osher_1988, title={Fronts propagating with curvature-dependent speed: Algorithms based on Hamilton-Jacobi formulations}, volume={79}, ISSN={0021-9991}, DOI={10.1016/0021-9991(88)90002-2}, number={1}, journal={Journal of Computational Physics}, publisher={Elsevier BV}, author={Osher, Stanley and Sethian, James A}, year={1988}, month=Nov, pages={12–49} }

@article{Osting_2019, title={A diffusion generated method for orthogonal matrix-valued fields}, volume={89}, ISSN={0025-5718}, DOI={10.1090/mcom/3473}, number={322}, journal={Mathematics of Computation}, publisher={American Mathematical Society (AMS)}, author={Osting, Braxton and Wang, Dong}, year={2019}, month=Sep, pages={515–550} }

@inproceedings{rieffel1967dentable,
  title={Dentable subsets of Banach spaces, with applications to a Radon-Nikodym theorem},
  author={Rieffel, MA},
  booktitle={Proc. Conf. Functional Analysis, Thompson Book Co., Washington, DC},
  pages={71--77},
  year={1967}
}

@article{merriman1992diffusion, title={Diffusion-Generated Motion by Mean Curvature for Filaments}, volume={11}, ISSN={1432-1467}, DOI={10.1007/s00332-001-0404-x}, number={6}, journal={Journal of Nonlinear Science}, publisher={Springer Science and Business Media LLC}, author={Ruuth, S. J. and Merriman, B. and Xin, J. and Osher, S.}, year={2001}, month=Dec, pages={473–493} }

@inbook{saye2020review, title={A Review of Level Set Methods to Model Interfaces Moving under Complex Physics: {{Recent}} Challenges and Advances}, ISBN={9780444640031}, ISSN={1570-8659}, DOI={10.1016/bs.hna.2019.07.003}, booktitle={Geometric Partial Differential Equations - Part I}, publisher={Elsevier}, author={Saye, Robert and Sethian, James A.}, year={2020}, pages={509–554} }

@article{stegall1978optimization, title={Optimization of functions on certain subsets of Banach spaces}, volume={236}, ISSN={1432-1807}, url={http://dx.doi.org/10.1007/bf01351389}, DOI={10.1007/bf01351389}, number={2}, journal={Mathematische Annalen}, publisher={Springer Science and Business Media LLC}, author={Stegall, Charles}, year={1978}, month=Jun, pages={171–176} }

@article{Sussman_1994, title={A Level Set Approach for Computing Solutions to Incompressible Two-Phase Flow}, volume={114}, ISSN={0021-9991}, DOI={10.1006/jcph.1994.1155}, number={1}, journal={Journal of Computational Physics}, publisher={Elsevier BV}, author={Sussman, Mark and Smereka, Peter and Osher, Stanley}, year={1994}, month=Sep, pages={146–159} }

@article{Takezawa_2010, title={Shape and topology optimization based on the phase field method and sensitivity analysis}, volume={229}, ISSN={0021-9991}, DOI={10.1016/j.jcp.2009.12.017}, number={7}, journal={Journal of Computational Physics}, publisher={Elsevier BV}, author={Takezawa, Akihiro and Nishiwaki, Shinji and Kitamura, Mitsuru}, year={2010}, month=Apr, pages={2697–2718} }

@article{van_Gennip_2014, title={Mean Curvature, Threshold Dynamics, and Phase Field Theory on Finite Graphs}, volume={82}, ISSN={1424-9294}, DOI={10.1007/s00032-014-0216-8}, number={1}, journal={Milan Journal of Mathematics}, publisher={Springer Science and Business Media LLC}, author={van Gennip, Yves and Guillen, Nestor and Osting, Braxton and Bertozzi, Andrea L.}, year={2014}, month=Apr, pages={3–65} }

@article{Wang_2022a, title={An Efficient Unconditionally Stable Method for Dirichlet Partitions in Arbitrary Domains}, volume={44}, ISSN={1095-7197}, DOI={10.1137/21m1443406}, number={4}, journal={SIAM Journal on Scientific Computing}, publisher={Society for Industrial & Applied Mathematics (SIAM)}, author={Wang, Dong}, year={2022}, month=Jul, pages={A2061–A2088} }

@article{Wang_2017, title={An efficient iterative thresholding method for image segmentation}, volume={350}, ISSN={0021-9991}, DOI={10.1016/j.jcp.2017.08.020}, journal={Journal of Computational Physics}, publisher={Elsevier BV}, author={Wang, Dong and Li, Haohan and Wei, Xiaoyu and Wang, Xiao-Ping}, year={2017}, month=Dec, pages={657–667} }

@article{Wang_2019a, title={Interface Dynamics for an Allen--Cahn-Type Equation Governing a Matrix-Valued Field}, volume={17}, ISSN={1540-3467}, DOI={10.1137/19m1250595}, number={4}, journal={Multiscale Modeling \& Simulation}, publisher={Society for Industrial & Applied Mathematics (SIAM)}, author={Wang, Dong and Osting, Braxton and Wang, Xiao-Ping}, year={2019}, month=Jan, pages={1252–1273} }

@article{Wang_2022b, title={The iterative convolution–thresholding method ({ICTM}) for image segmentation}, volume={130}, ISSN={0031-3203}, DOI={10.1016/j.patcog.2022.108794}, journal={Pattern Recognition}, publisher={Elsevier BV}, author={Wang, Dong and Wang, Xiao-Ping}, year={2022}, month=Oct, pages={108794} }

@article{Wang_2019b, title={An improved threshold dynamics method for wetting dynamics}, volume={392}, ISSN={0021-9991}, DOI={10.1016/j.jcp.2019.04.037}, journal={Journal of Computational Physics}, publisher={Elsevier BV}, author={Wang, Dong and Wang, Xiao-Ping and Xu, Xianmin}, year={2019}, month=Sep, pages={291–310} }

@article{Xu_2017, title={An efficient threshold dynamics method for wetting on rough surfaces}, volume={330}, ISSN={0021-9991}, DOI={10.1016/j.jcp.2016.11.008}, journal={Journal of Computational Physics}, publisher={Elsevier BV}, author={Xu, Xianmin and Wang, Dong and Wang, Xiao-Ping}, year={2017}, month=Feb, pages={510–528} }

@article{Xu_2019, title={Energy Stable Semi-implicit Schemes for Allen–Cahn–Ohta–Kawasaki Model in Binary System}, volume={80}, ISSN={1573-7691}, DOI={10.1007/s10915-019-00993-4}, number={3}, journal={Journal of Scientific Computing}, publisher={Springer Science and Business Media LLC}, author={Xu, Xiang and Zhao, Yanxiang}, year={2019}, month=Jun, pages={1656–1680} }

@article{Zhang_2026, title={Up to Fourth-Order Structure-Preserving Diagonally Implicit Runge-Kutta Schemes for the Allen-Cahn-Type Equations}, volume={18}, ISSN={2070-0733}, DOI={10.4208/aamm.oa-2025-0119}, number={6}, journal={Advances in Applied Mathematics and Mechanics}, publisher={Global Science Press}, author={Zhang, Wei and Zhu, Xiaohan and Cai, Wenjun and Wang, Yushun}, year={2026}, month=Jun, pages={1914–1938} }
